\documentclass[pdflatex,sn-mathphys-num]{sn-jnl}

\usepackage{graphicx}%
\usepackage{multirow}%
\usepackage{amsmath,amssymb,amsfonts}%
\usepackage{amsthm}%
\usepackage{mathrsfs}%
\usepackage[title]{appendix}%
\usepackage{xcolor}%
\usepackage{xfrac}%
\usepackage{textcomp}%
\usepackage{manyfoot}%
\usepackage{booktabs}%
\usepackage{algorithm}%
\usepackage{algorithmicx}%
\usepackage{algpseudocode}
\usepackage{subfigure}
\usepackage{algpseudocode}%
\usepackage{listings}%
\definecolor{mygreen}{rgb}{0,0.6,0}
\definecolor{mygray}{rgb}{0.5,0.5,0.5}
\definecolor{mymauve}{rgb}{0.58,0,0.82}
\definecolor{altblue}{rgb}{0.0,0.6,1.0}
\definecolor{lstbg}{cmyk}{0.05, 0.01, 0, 0}
\definecolor{morebluish}{cmyk}{0.06,0.04,0,0}
\input{listings-maple-definition.sty}
\newcommand{\e}{{\varepsilon}}
\newcommand{\snot}[2]{{{#1}\times 10^{#2}}} 

\newcommand{\yWKB}{{y_{\mathrm{WKB}}}}
\newcommand{\yref}{{y_{\mathrm{ref}}}}

\theoremstyle{thmstyleone}%
\newtheorem{theorem}{Theorem}
\theoremstyle{thmstyletwo}%
\newtheorem{example}{Example}%
\newtheorem{remark}{Remark}%

\theoremstyle{thmstylethree}%

\begin{document}

\title[Structured backward error for WKB]{Structured backward error for the WKB method}


\author*[1]{\fnm{Robert M.} \sur{Corless}}\email{rcorless@uwo.ca}

\author[2]{\fnm{Nicolas} \sur{Fillion}}\email{nfillion@sfu.ca}
\equalcont{These authors contributed equally to this work.}

\affil*[1]{\orgdiv{The Department of Computer Science, and the Rotman Institute of Philosophy}, \orgname{Western University}, \orgaddress{\street{1151 Richmond St}, \city{London}, \postcode{N6A 3K7}, \state{Ontario}, \country{Canada}}}

\affil[2]{\orgdiv{Department of Philosophy}, \orgname{Simon Fraser University}, \orgaddress{\street{8888 University Drive}, \city{Burnaby}, \postcode{V5A 1S6}, \state{British Columbia}, \country{Canada}}}


\abstract{The classical WKB method (also known as the WKBJ method, the LG method, or the phase integral method) for solving singularly perturbed linear differential equations has never, as far as we know, been looked at from the structured backward error (BEA) point of view.  This is somewhat surprising, because a simple computation shows that for some important problems, the WKB method gives the exact solution of a problem \textsl{of the same structure} that can be expressed in finitely many terms. This kind of analysis can be extremely useful in assessing the validity of a solution provided by the WKB method.  In this paper we show how to do this and explore some of the consequences, which include a new iterative algorithm to improve the quality of the WKB solution.  We also explore a new hybrid method where the potential is approximated by Chebyshev polynomials, which can be implemented in a few lines of Chebfun.}



\maketitle


\section{Introduction\label{sec:intro}}
The classical WKB method\footnote{The name WKB comes from the initials of three of the researchers who worked on the problem in the early part of the twentieth century, Wentzel, Kramers, and Brillouin.  Other names for this method are the LG method (for Liouville and Green; Smith uses this in~\cite{Smith1985}) and yet another name is the WKBJ method (where the J is added for Jeffreys). Dingle makes the case in~\cite{Dingle1972} that the name really should be the ``phase integral method.'' Short of following this recommendation, `the WKB method' is at least historically coherent.} is intended to approximately solve linear differential equations in which a small parameter is multiplying the highest derivative, such as the following class, often called Schr\"odinger-type equations:
\begin{equation}\label{eq:WKBSchroedinger}
  \e^2 y'' = Q(x) y\>.
\end{equation}
Equations of this type frequently come up in classical physics (as instances of Newton's second law) and in quantum physics (as instances of the Schr\"odinger equation), among many others. Due to the prevalence of equations of this form in physics, the expression $Q(x)y$ is often called a `potential,' although we will also take the liberty of sometimes referring to $Q(x)$ itself as a potential. Since the WKB method produces a simple formula for the approximate solution of equation~\eqref{eq:WKBSchroedinger} that is in many cases quite accurate, the technique has become quite popular. 
Furthermore, when $\e$ is small, numerical methods can be slow to solve the problem accurately, so that it can be beneficial to have a complementary perturbative technique.

Following its classical description, the WKB method posits a solution of the form
\begin{equation}\label{eq:wqbPosit}
    y_{\sc\rm WKB} = \exp\left( \frac{1}{\delta} S(x) \right) \>,
\end{equation}
where $S(x)$ has the possibly infinite\footnote{We think that even writing this potentially infinite series down is a bit of a misstep, and we hope this paper shows why.} series representation
\begin{equation}\label{eqSx}
    S(x) = S_0(x) + \delta S_1(x) + \delta^2S_2(x) +\cdots \>.
\end{equation}
With an appropriate choice of $\delta$, solving for the $S_k$s is expected to produce a good approximation to $y$, provided that $Q(x)$ is a ``slowly varying function.'' In his historical survey of the method, \citet{schlissel1977initial} remarks that ``[t]he various researchers [including the likes of Green, Rayleigh, Gans, Jeffreys, Wentzel, Kramers, and Brillouin] attempted to justify their procedures, but these were at best token gestures,'' since ``[t]he contributors were in many instances physicists searching to obtain approximate solutions whose nature they had predetermined'' (p.~184). Further historical remarks about the WKB method are given in~\cite{CorlessFillion2025}.  Nowadays, there exists rigorous mathematical treatments of the method (see, e.g.~\cite{Smith1976} or for a more undergraduate level presentation~\cite{Murdock1999}); this paper seeks to supplement those rigorous treatments by using a different perspective on error analysis, namely, backward error analysis (BEA), that in many respects echoes the pioneering physicists' way of thinking. 

\textsl{Both} of the just-cited references come remarkably close to the presentation in this paper, but \textsl{just} miss: neither one gives a structured backward error interpretation, even though~\cite{Smith1976} uses residuals and~\cite{Murdock1999} even gives formulas equivalent to the ones we derive here.

Indeed, to our knowledge, this paper and the book~\cite{CorlessFillion2025} are the first to examine the WKB method from the point of view of structured BEA. Somewhat surprisingly to us, the WKB method has, for Schr\"odinger-type equations, a very simple interpretation in terms of {\sl structured} backward error, which does not seem to have been noticed before\footnote{Gans at least computed a residual, in his 1915 paper~\cite{gans1915fortpflanzung}. However, he did not compute a \textsl{relative} residual and thus did not notice the structured backward error interpretation.  In any case, such interpretations did not come into the literature until the 1950s, at least.}. In our view, this particular structured backward error analysis deserves much wider attention, because this interpretation really explains why the method works so well when it does, and demonstrates convincingly what goes wrong when it doesn't. 

The interpretation also shows why the standard analysis of the conditioning of the problem, using the approximate Green's function obtained using the WKB method, gives an excellent unstructured condition number for the problem.  In addition, one can deduce a \textsl{structured} condition number by direct examination of the symbolic answer. Such backward error plus conditioning analyses are common in numerical linear algebra, but less common in numerical solution of differential equations, although not unknown there~\citep[see, e.g.,][]{CorlessFillion(2013)}. 

We do not know when the idea of using BEA (plus conditioning) to interpret perturbation methods was first proposed, but one can see it in~\cite{Smith1976} at least, and in~\cite{Roberts(2014)}, as well as in \cite{corless2019backward}. We also feel that the backward error idea, though ``well-known'' in the numerical analysis community, itself deserves wider attention.
 
\subsection{Organization of the paper}
This paper extends the treatment of this idea present in~\cite{CorlessFillion2025}.  We give some background from that book in section \ref{subsec:bg} in order to make this paper self-contained.  
Subsections \ref{sec:optics},  \ref{sec:green}, and \ref{sec:IterativeWKB}  of section \ref{subsec:bg} summarize material from our book, using different examples that we believe are more appropriate for this audience.
We give a structured backward error interpretation of the method's solution to a simple but important class of problems. 
We mention an unstructured conditioning analysis by means of Green's functions. 
We also show how the standard WKB formula can be usefully recast as an iterative algorithm.  

In the main part of the paper, namely section~\ref{sec:hybrid}, we discuss a new hybrid WKB algorithm which uses Chebyshev approximation of the square root $\sqrt{Q(x)}$ as one technique to eliminate a bottleneck in the WKB method.  We show how this works in both Maple and in Chebfun~\cite{battles2004extension} in MATLAB.


\section{Background from our book}\label{subsec:bg}
The first thing is that our book~\cite{CorlessFillion2025} treats perturbation methods as an \textsl{iteration}. This approach has several advantages, the first of which being that we never posit an infinite (possibly divergent) series as is the case with the classical account of the method based on equations~\eqref{eq:wqbPosit} and~\eqref{eqSx}. Instead, we simply consider functions of the finite form
\begin{equation}\label{eq:finiteansatz}
    y_n(x) = \textrm{exp}\left(\frac{1}{\delta}\sum_{k=0}^n S_k(x)\delta^k \right)
\end{equation}
together with their residuals in the differential equation. With respect to equation~\eqref{eq:WKBSchroedinger}, the residual is given by 
\begin{equation}\label{residSchrodinger}
    r(y_n) = \e^2y_n''-Q(x)y_n\>,
\end{equation}
and the relative residual is $r(y_n)/y_n = \e^2y_n''/y_n - Q(x)$. 

More generally, for a differential equation of the form $F(x,y;\e)=0$, the residual of $y_n$ is given by $F(x,y_n;\e)$ (we usually simply write $F(y_n)$ for simplicity). In this way, we only work with finite constructions and we know immediately if the approximations do not improve, because we can see if the residuals $F(y_n)$ decrease or not.  If the size of $F(y_{n+1})$ is not smaller than the size of $F(y_n)$, the iteration has \textsl{stagnated}.  Working in series, this typically only happens with the initial approximation $y_0$, and then only if the initial approximation is not good enough.  

This general abstract perturbation method starts with an initial approximation $y_0$, and then (essentially) uses functional iteration based on Newton's method:
\begin{equation}
    y_{n+1} = y_n - \frac{F(y_n)}{F'(y_0)}\>. 
\end{equation}
Here the derivative is a Fr\'echet derivative and the inversion in the notional equation above represents solving the linearized problem---linearized about the initial approximation---to get one more term in the expansion. Because the linearization only happens once, this is functional iteration, and converges linearly at best if done numerically, but each iteration gives one more correct term in a perturbation expansion if done in series.

A second advantage of viewing perturbation as an iteration is that the residual $F(y_n)$ is computed every time, frequently in finite terms, and we have at each step the \textsl{exact} solution to the equation $F(y)-F(y_n) = 0$. This is both trivial and profound, as it tells us one of the modified equations that $y_n$ solves exactly (or, in the terminology of \cite{CorlessFillion(2013)}, what reversed-engineered problem $y_n$ solves). The residual $F(y_n)$ is a kind of backward error (other kinds might be computed, as we will see). 

Putting the emphasis on the question ``what equation does $y_n$ solve exactly?'' is an alternative to the more traditional focus on the question ``does $F(y_n)\approx 0$'' that we contend leads to more rigorous error analysis and more readily interpretable results.

This change of focus leads to a third advantage of this approach, namely, that it is then natural to compute a final residual and interpret it in terms of the original model.  Once that is done, we may use the standard theory of conditioning to understand if our solution is a good one or not. 

Interestingly, \cite{schlissel1977initial} describes how, in 1915, Gans used the WKB method (11 years prior to the papers of Wentzel, Kramers, and Brillouin) to study light propagation in inhomogenous media governed by equation~\eqref{eq:WKBSchroedinger}, and computed the residual to argue that the $y_n$s given by the WKB method ``almost satisfy'' equation~\eqref{eq:WKBSchroedinger}. Schlissel criticizes this justification, but one only needs to further consider the  conditioning, as we do, to complete Gans' argument. Indeed, our paper shows that Gans' strategy to assess the error of the $y_n$s would have been adequate had he had the framework of backward error analysis at hand to interpret the residual. Of course, it is only with Wilkinson's work in the 1960s that the framework became widely understood~\citep[see, e.g.,][]{wilkinson1971modern}, so that wasn't an option for Gans.

\subsection{The approximation from physical optics} \label{sec:optics}
The following gives the basic WKB formula, which is derived in several places, including~\cite{Bender(1978)}.  The derivation starts with the \textsl{finite} sum $S_0(x)/\e + S_1(x)$. This formula is sometimes known as the \textsl{approximation from physical optics}:
\begin{equation}\label{eq:WKBform}
  \yWKB(x) = c_1 Q(x)^{-1/4} e^{S_0(x)/\e} + c_2 Q(x)^{-1/4} e^{-S_0(x)/\e}
\end{equation}
where
\begin{equation}\label{eq:QWKB}
  S_0(x) = \int_{0}^{x} \sqrt{Q(\xi)}\,d\xi\>.
\end{equation}
The lower limit is unimportant, although it can be chosen to make some computations more convenient.
The function $\yWKB$ is supposed to be a reasonable approximation, for small $\e$, of the solution to equation~\eqref{eq:WKBSchroedinger}. 
In section \ref{sec:IterativeWKB}, we will use this as an initial approximation in an apparently novel iterative perturbation scheme. Within that scheme, as we will see, this reasonably simple approach yields quite accurate initial approximations\index{initial approximation}. But first, we examine a suggestive example.

\begin{example}
Let's take $\e^2y'' + (1+x^8)y = 0$ as an instance of equation~\eqref{eq:WKBSchroedinger}, so that $Q(x) = -(1+x^8)$. 

Applying equation~\eqref{eq:QWKB}, we first compute
\begin{equation}\label{eq:WKBexampleS0}
  S_0 = \int_{0}^{x} \sqrt{1+\xi^8}\,d\xi = x F\left(\left. { -\sfrac12, \sfrac18} \atop {1+\sfrac18}\right| -x^8 \right)
\end{equation}
where $F$ is a hypergeometric function. In Maple (similarly in MATLAB) that expression is \lstinline{x*hypergeom([-1/2,1/8],[9/8],-x^8)}.
Following equation~\eqref{eq:WKBform}, our WKB approximation is 
\begin{equation}\label{eq:WKBexampley}
\yWKB(x) = \frac{c_{1} \cos \! \left(\sfrac{S_{0}}{\varepsilon}\right)}{\left(x^{8}+1\right)^{\frac{1}{4}}}+\frac{c_{2} \sin \! \left(\sfrac{S_{0}}{\varepsilon}\right)}{\left(x^{8}+1\right)^{\frac{1}{4}}}
\end{equation}
where $S_0$ is as above.
Here $c_1$ and $c_2$ are arbitrary constants. We can identify the constants by fitting boundary conditions, for nonexceptional values\footnote{The exceptional values are eigenvalues of the problem, which we ignore in this paper.} of $\e$.  If, for instance, $y(-1)=1$ and $y(1)=2$, then we get certain numerical values that depend on $\e$ for each of $c_1$ and $c_2$.

But even before we apply the boundary conditions, we compute the residual\index{residual!WKB example} given by equation~\eqref{residSchrodinger}, and---this is important---notice that $\yWKB$ is a factor of the residual:
\begin{equation}\label{eq:residualWKB}
  r(x) = \e^2\yWKB''(x) - (1+x^8)\yWKB(x) = \e^2 
\frac{2 x^{6} \left(3 x^{8}-7\right)}{\left(x^{8}+1\right)^{2}}\yWKB\>.
\end{equation}
Of course we did that in Maple, although the simplification needed some human help. 

For the computed residual to be enlightening, we interpret it as follows: we have found the \textsl{exact} solution to the Schr\"odinger-like equation $\e^2 y'' + \widetilde{Q}(x) y = 0$ where
\begin{equation}\label{eq:structuredbackwarderrorWKB}
  \widetilde{Q}(x) = 1+x^8 + \e^2 Q_2(x) = 1 + x^8 + \e^2 \frac{2 x^{6} \left(3 x^{8}-7\right)}{\left(x^{8}+1\right)^{2}}\>.
\end{equation}
We plot $Q_2(x)$ in figure~\ref{fig:PerturbedPotential} and $y(x)$ for $\e=1/13$ in figure~\ref{fig:hypergeometricG13}(a).

Let us emphasize the point of interpreting the residual as we've suggested: The WKB method has found the \textsl{exact} solution of
a problem of the same type as the one given by equation~\eqref{eq:WKBSchroedinger}, with a potential $\widetilde{Q}(x)$ different by an order of $\e^2$.  
\begin{figure}
  \centering
  \includegraphics[width=0.7\textwidth]{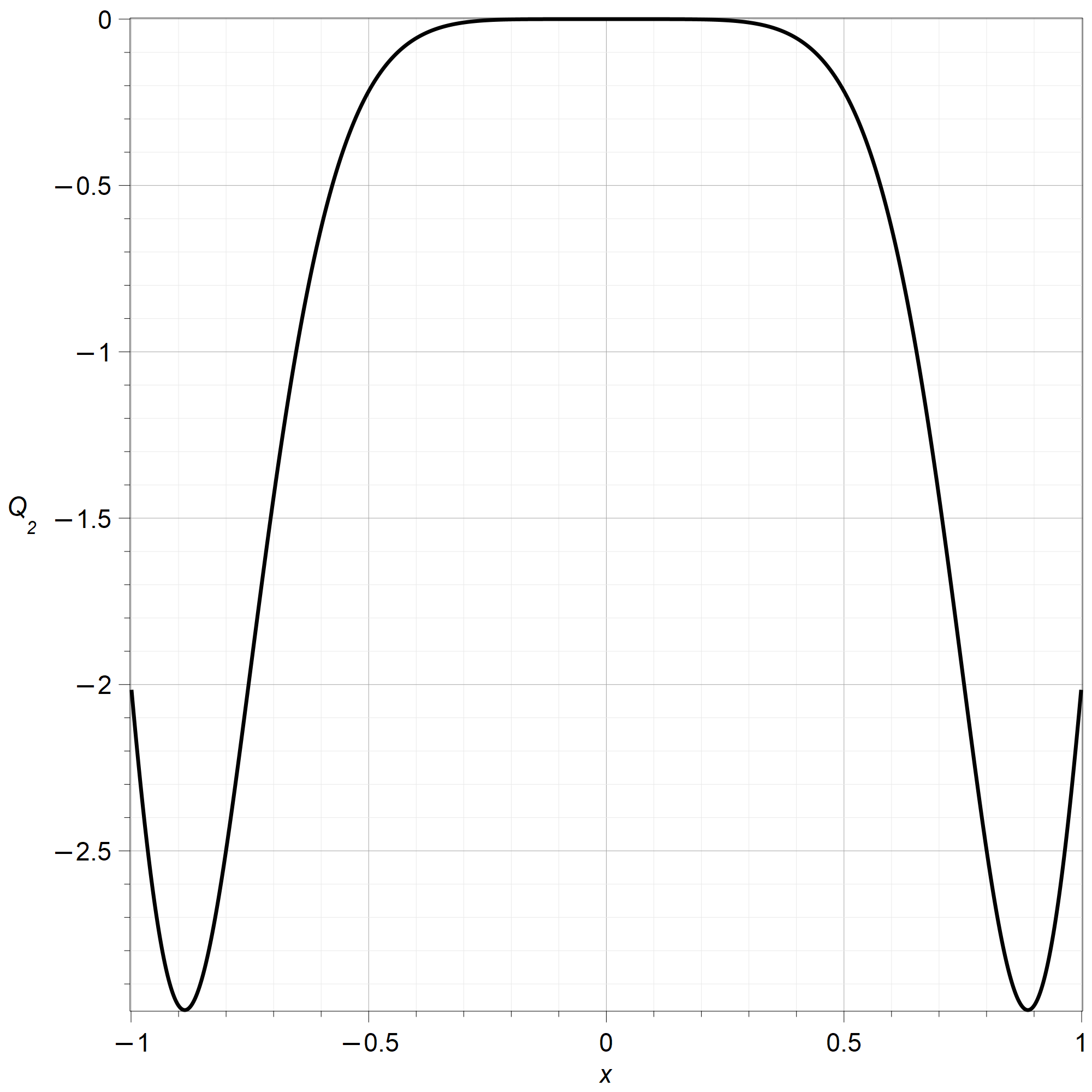}
  \caption[Perturbed potential from WKB]{The perturbation $Q_2(x) = \frac{2 x^{6} \left(3 x^{8}-7\right)}{\left(x^{8}+1\right)^{2}}$ to the potential $Q(x)=1+x^8$ from equation~\eqref{eq:structuredbackwarderrorWKB}. We see that on the interval of interest, $\e^2 Q_2(x)$ will be less than $3\e^2$ in magnitude. }\label{fig:PerturbedPotential}
\end{figure}
This perturbed potential\index{potential} can be interpreted in the physical terms of the original model. We also see that this particular perturbation is uniformly bounded in magnitude for all  $x$ in $[-1,1]$ by $3\e^2$.

Note that when $-1 < x < 1$, $1+x^8$ is quite close to $1$ except near the endpoints.  Thus the potential is nearly constant, and so we would expect nearly pure harmonic oscillation.  That is what we observe in figure~\ref{fig:hypergeometricG13}(a).  Outside that interval, the potential is quite different, and the solution is also quite different.  These observations help to internalize the idea of looking at the backward error.
\end{example}

The behaviour displayed in this example and the style of analysis we used is quite general, as far as equations of the form of \eqref{eq:WKBSchroedinger} are concerned, as the next theorem shows.

\begin{theorem}[The Backward WKB Theorem] \label{thm:BackwardWKB} \index{WKB Backward Theorem} 
If $Q(x)$ is twice continuously differentiable on $a < x < b$ (either $a$ or $b$ or both may be infinite), and $Q(x) \ne 0$ on $a < x < b$, then the approximation from physical optics in equations~\eqref{eq:WKBform}--\eqref{eq:QWKB} is the \textsl{exact} general solution to
\begin{equation}\label{eq:WKBbackward}
  \e^2 y'' = \widetilde{Q}(x) y\>,
\end{equation}
where
\begin{equation}\label{eq:TildeQgeneral}
\widetilde{Q}(x) = Q(x) + \e^2 Q_2(x) = Q(x) + \e^2 \left(  5\left(\frac{Q'(x)}{4Q(x)}\right)^2 - \frac{Q''(x)}{4Q(x)} \right) \>.
\end{equation}
\end{theorem}
\begin{proof}
Take the logarithm of the ansatz in equation~\eqref{eq:finiteansatz} and differentiate:
\begin{align}
    \ln y_n &= \frac{1}{\delta} \sum_{k=0}^n S_k(x)\delta^k \\
    \frac{y_n'}{y_n} &= \frac{1}{\delta} \sum_{k=0}^n S_k'(x)\delta^k 
\end{align}
Thus we have $y_n' = \mathcal{S}'(x) y_n$, where $\mathcal{S}'$ is the finite sum
\begin{equation}
    \mathcal{S}' = \frac{1}{\delta} \sum_{k=0}^n S_k'(x)\delta^k\>.
\end{equation}
Then $y_n'' = \mathcal{S}''y_n + \mathcal{S}'y_n' = \left(\mathcal{S}'' + \left(\mathcal{S}'\right)^2\right)y_n$, which has a finite number of terms multiplying $y_n$.

In the case of the second-order equation under consideration, $\e^2 y'' = Q(x)y$, the choice $\delta = \e$ is the right one, as found for example by a Newton polygon. That means that the leading term in $\mathcal{S}'$ is $S_0'/\e$, and the leading term in $\left(\mathcal{S}'' + \left(\mathcal{S}'\right)^2\right)$ is $\left(S_0'\right)^2/\e^2$. Therefore $S_0'(x) = \pm\sqrt{Q(x)}$ (either sign) as usual, and setting the next power of $\e$ to zero (to match the right hand side further) gives $S_0'' + 2S_1'S_0' = 0$, again as usual, so $S_1' = -S_0''/(2S_0') = -Q'/(4Q)$ and therefore $S_1 = |Q(x)|^{-1/4}$. We will normally wave the absolute value signs away, because if no turning point is encountered, the sign of $Q(x)$ can be dealt with by adjusting the constants $c_1$ and $c_2$, and we assume that the reader will look after that detail.

Because $n$ is finite, the remaining expression also contains only finitely many terms. In the simple case when $n=1$, we have $\mathcal{S} = S_0/\e + S_1$ and so
\begin{equation}
    \e^2 \mathcal{S}'' + \e^2\left(\mathcal{S}'\right)^2 = Q(x) + \e^2\left(S_1'' + \left(S_1')\right)^2\right)\>. 
\end{equation}
Substituting $S_1' = -Q'/(4Q)$ in that expression yields
\begin{equation}
Q(x) + \e^2 Q_2(x) =    Q(x) + \e^2\left( 5\left( \frac{Q'}{4Q}\right)^2 - \frac{Q''}{4Q} \right)\>.
\end{equation}
That is, the WKB solution $y = \exp(  S_0/\e + S_1)$ with $S_0' = \pm \sqrt{Q}$ and $S_1' = -Q'/(4Q)$ gives the \textsl{exact} solution to $\e^2 y'' = (Q(x) + \e^2Q_2(x)y$ where $Q_2$ is defined as above.

The key point is that the expression for $Q_2$ is the \textsl{same} whatever sign we take for $S_0$, so we have that the residual of $y_1 = Q^{-1/4} \exp( S_0/\e )$ is $\e^2 Q_2 y_1$ and the residual of $y_2 = Q^{-1/4} \exp(-S_0/\e )$ is $\e^2 Q_2 y_2$. Therefore the residual of $\yWKB = c_1 y_1 + c_2 y_2$ is $c_1 \e^2 Q_2 y_1 + c_2 \e^2 Q_2 y_2$ or $\e^2 Q_2 \yWKB$. 
\end{proof}

As long as $Q(x)$ is not zero, and $Q(x)$ has a bounded second derivative, the residual $\e^2 Q_2$ is going to be a \textsl{small} perturbation of the potential when $\e$ is small enough.  This means that the WKB approximation gets us the exact solution to $\e^2 y'' = \widetilde{Q}(x)y$, where $\widetilde{Q}(x) = Q(x) + O(\e^2)$, and moreover we have a simple, finite formula for that perturbation.

This will be useful in at least two ways. One is that we can interpret this perturbation in terms of the original model. 
If the extent of the uncertainty modelers have with respect to $Q$ is larger than $\e^2Q_2$ is, then, for all they know, this WKB solution might exactly describe the system. As such, the WKB solution would be as good as one could hope to obtain given what is known and unknown about the system of interest. 

Secondly, we can feed the negative of this residual back into the WKB process to produce an $O(\e^4)$ accurate solution, as measured by the backward error.  We will examine this idea in section~\ref{sec:IterativeWKB}. But first, we briefly discuss the assessment of the conditioning.

\subsection{Green's Functions \label{sec:green}}
The standard theory of Green's functions can be used to compute the condition number of the differential equation.  The purpose is not to estimate the forward error of the computation, however, although it can be used for that as well.  The purpose is to learn how sensitive the equation is to changes.  We find that the Green's function $G(x,\xi)$ for this class of problem is always of $O(1/\e)$, meaning that it is sensitive.  We leave the details to~\cite{CorlessFillion2025}, as this topic is well-understood. The main point for this paper is that these Green's functions are the \textsl{exact} Green's functions for the perturbed problem with $\widehat{Q}(x) = Q(x) + \e^2 Q_2(x)$.  This knowledge allows firm conclusions to be drawn.
\begin{example}
We continue with $\e^2 y'' + (1+x^8)y = 0$ subject to $y(-1)=2$, $y(1)=1$.  The boundary conditions for the Green's function are thus $G(-1,\xi)=0$, $G(1,\xi)=0$, $G(\xi^+,\xi)=G(\xi^-,\xi)$, and $G_x(\xi^+,\xi)-G_x(\xi^-,\xi) = -1/\e^2$. Maple gets a finite expression containing hypergeometric functions for $G(x,\xi)$ which we do not print here. We plot its contours, for $\e=1/13$, in figure~\ref{fig:hypergeometricG13}(b).
\begin{figure}
    \centering
    \subfigure[WKB Solution]{\includegraphics[width=0.48\textwidth]{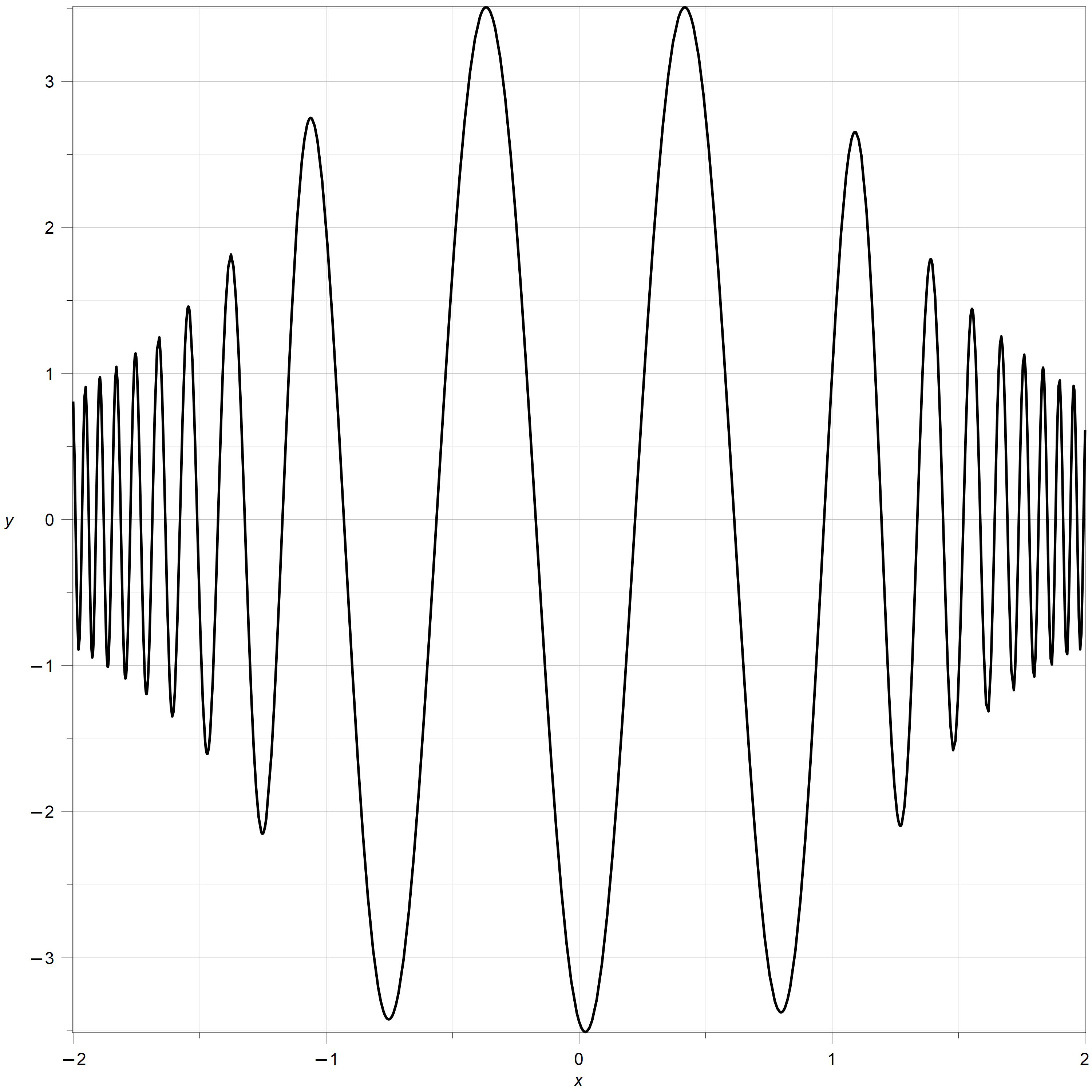}}
    \subfigure[Green's function]{\includegraphics[width=0.48\linewidth]{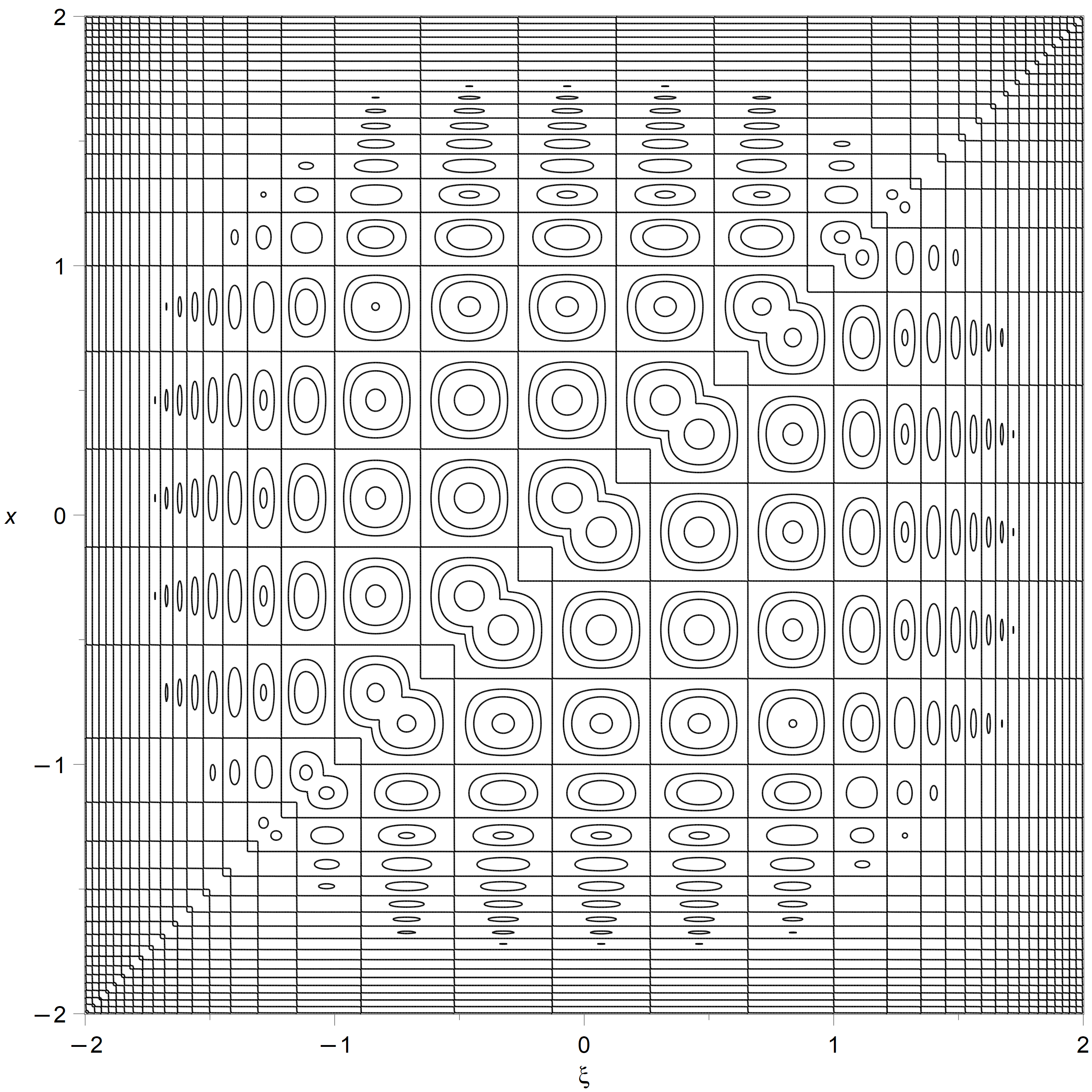}}
    \caption{(a) Solution by WKB of $\e^2 y'' + (1+x^8)y = 0$ subject to $y(-1)=2$, $y(1)=1$ when $\e=1/8$ (b) Contours for the Green's function for the same problem with the same $\e=1/8$. The contours are at $[-8, -5, -3, 0, 3, 5, 8]$. The maximum height of $G$ is $O(1/\e)$, as claimed. The residual is $\e^2 Q_2(x)$ where $Q_2(x) = 2 x^{6} \left(3 x^{8}-7\right)/(1+x^8)^2$. On $-2 \le x \le 2$, $Q_2(x)$ is no bigger than $3$ in magnitude.}
    \label{fig:hypergeometricG13}
\end{figure}
Then the difference between the reference solution to the original problem and the WKB solution can be written as
\begin{equation}
    \yWKB(x) - \yref(x) = \int_{x_0}^x G(x,\xi) \e^2 Q_2(\xi)\yref(\xi)\,d\xi
\end{equation}
Since $\|G\| = O(1/\e)$ the forward error is $O(\e)$ on compact intervals.
\end{example}

\subsection{Iterative WKB\label{sec:IterativeWKB}}
Section \ref{sec:optics} has shown how to perform a backward-error analysis 
of the WKB method as it applies to equation \ref{eq:WKBSchroedinger}. Section~\ref{sec:green} indicated how to understand the effects of such backward error using Green's functions.  But since this approach requires one to compute the residual and find which equation $\e^2y''=\widetilde{Q}(x)y$ we have solved, the idea of compensating for the difference between $Q(x)$ and $\widetilde{Q}(x)$ ahead of time emerges. This section explores the potential benefits of this idea.

\begin{example}
Let's consider iterating the WKB process, on a smooth example with $Q(x) = 1+x^2$. 
The WKB process will give us the exact answer not for that $Q$, but rather for $\widetilde{Q} = 1 + x^2 + \e^2 \left({ \left(3 x^{2}-2\right)}\right)/\left({4 \left(x^{2}+1\right)^{2}}\right)$. Here $Q(x)$ is analytic, not just twice continuously differentiable.  So what would happen if we tried instead to solve $\e^2y'' = \widehat{Q}(x)y$, where
\begin{equation}\label{eq:modifiedWKBex}
  \widehat{Q}(x) = Q(x) \textcolor{red}{-} \e^2 Q_2(x) = 1 + x^2 \textcolor{red}{-} \e^2 \frac{ \left(3 x^{2}-2\right)}{4 \left(x^{2}+1\right)^{2}}\>?
\end{equation}
Notice the deliberate opposite sign.  We are subtracting off ahead of time what is put in by the WKB process when applied to equation \ref{eq:WKBSchroedinger}.\index{WKB!subtracting off what will be put in}

The solution process needs the integral of the square root of $Q-\e^2 Q_2$, and we can do this perturbatively:
\begin{equation}\label{eq:intsqrtQ2}
  \int_{0}^{x} \sqrt{Q(\xi) - \e^2Q_2(\xi)}\,d\xi = \int_{0}^{x} \sqrt{Q(\xi)}\,d\xi - \e^2 \int_{0}^{x} \frac{Q_2(\xi)}{2\sqrt{Q(\xi)}}\,d\xi + O(\e^4)
\end{equation}
That second integral becomes the factor
\begin{equation}\label{eq:S2fac}
  F_2 = \exp\left( \e \frac{x \left(x^{2}+6\right)}{24 \left(x^{2}+1\right)^{\sfrac{3}{2}}} \right).
\end{equation}
Our modified solution will then be
\begin{equation}
    y=c_1 \widehat{Q}(x)^{-1/4} \exp( S_0/\e ) F_2 + c_2 \widehat{Q}(x)^{-1/4} \exp(-S_0/\e)/F_2
\end{equation}
for some arbitrary constants $c_1$ and $c_2$.  When we compute the residual of \textsl{this} solution, we find that this function is the exact solution of $\e^2y'' = Q^*(x)y$ where $Q^*(x) = Q(x) + \e^4 Q_4(x) + O(\e^5)$, where
\begin{equation}\label{eq:fourthorder}
  Q_4(x) = \frac{297 x^{4}-732 x^{2}+76}{64 \left(x^{2}+1\right)^{5}}\>.
\end{equation}
It turns out that this behaviour is quite general as far as solutions to equation \eqref{eq:WKBSchroedinger} are concerned, as we will see.  
See figure~\ref{fig:PerturbedPotential4th}.
\begin{figure}
  \centering
  \subfigure[From $Q=1+x^2$]{\includegraphics[width=0.45\textwidth]{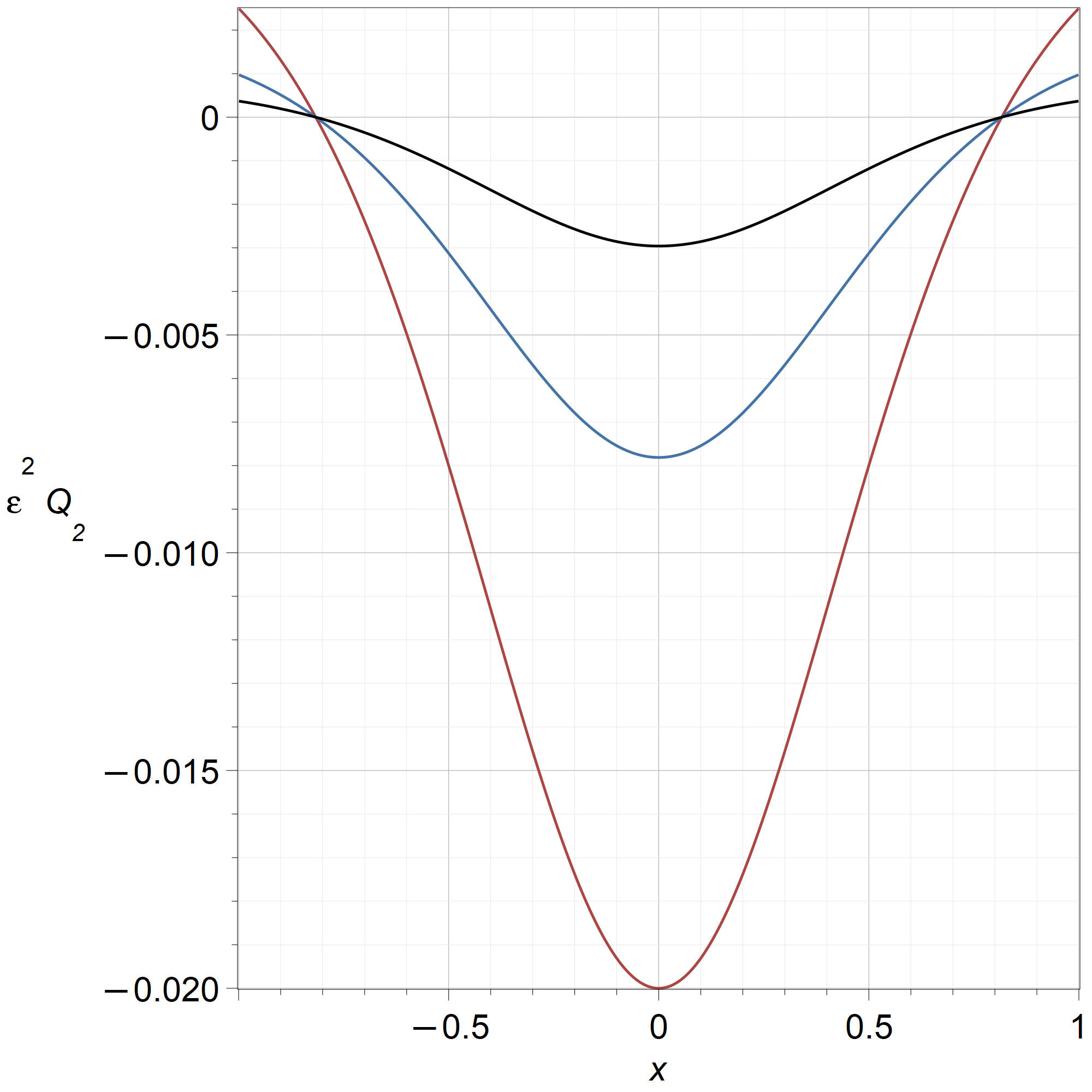}}
  \subfigure[From $Q=1+x^8$]{\includegraphics[width=0.45\textwidth]{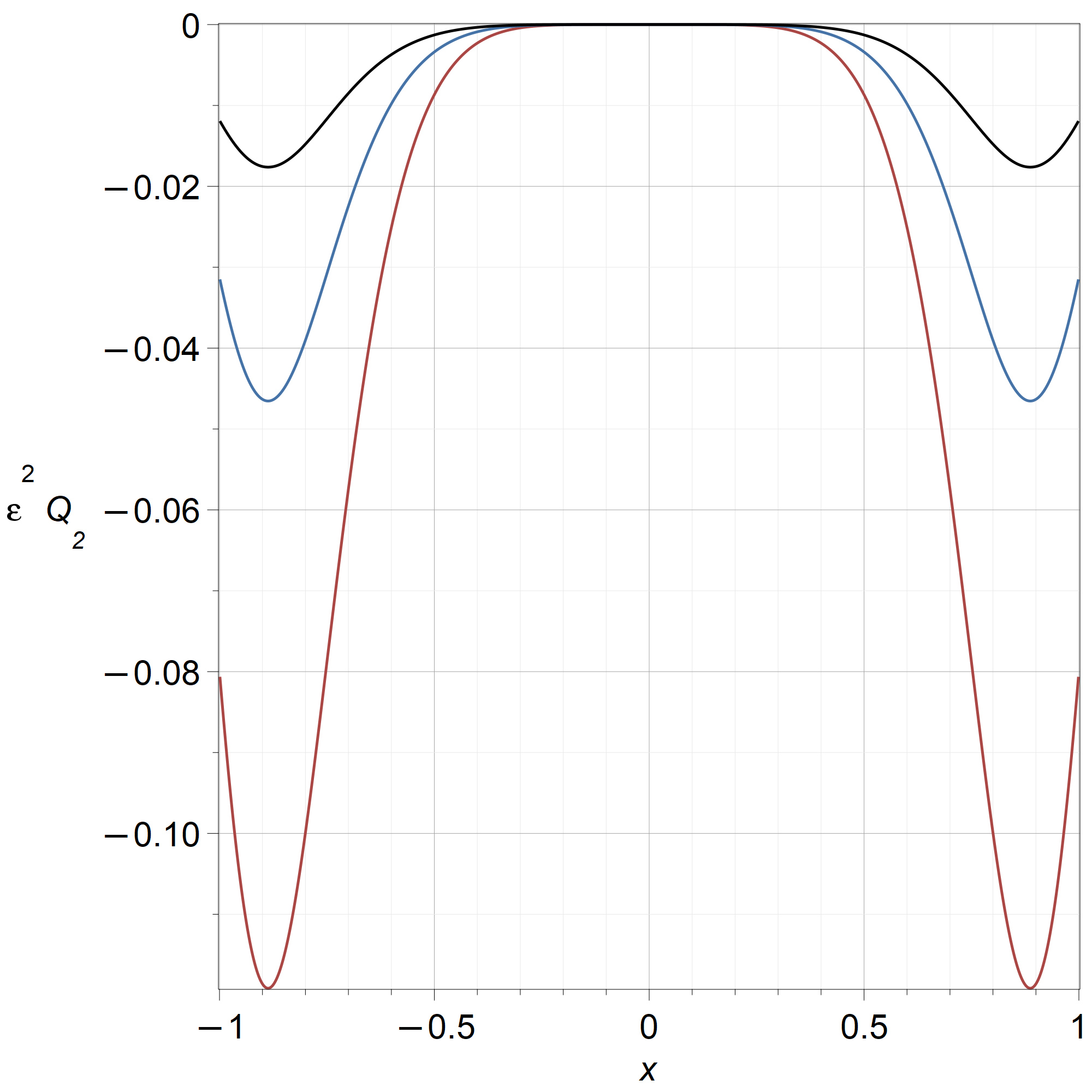}}\\  
  \subfigure[From $Q=1+x^2$]{\includegraphics[width=0.45\textwidth]{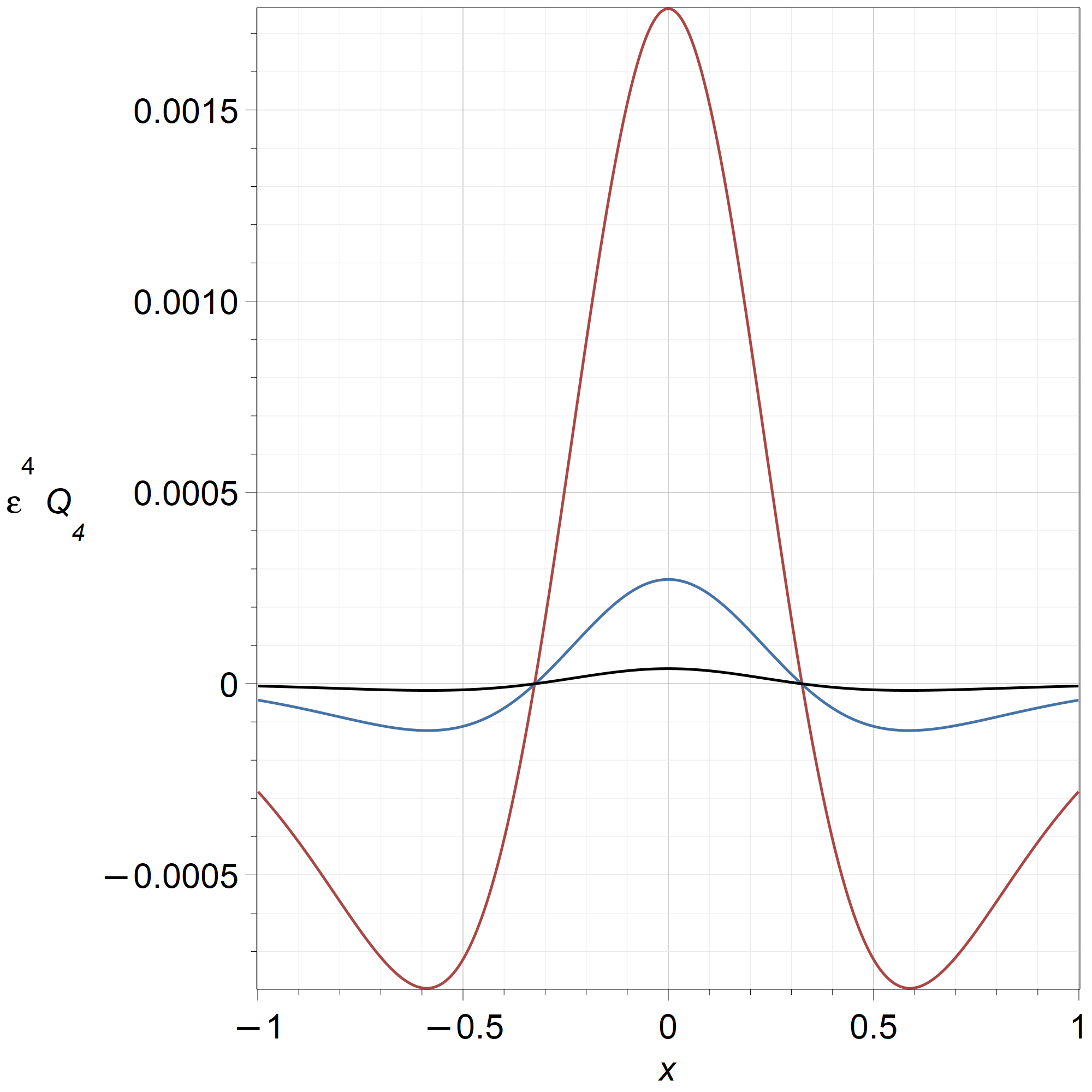}}
  \subfigure[From $Q=1+x^8$]{\includegraphics[width=0.45\textwidth]{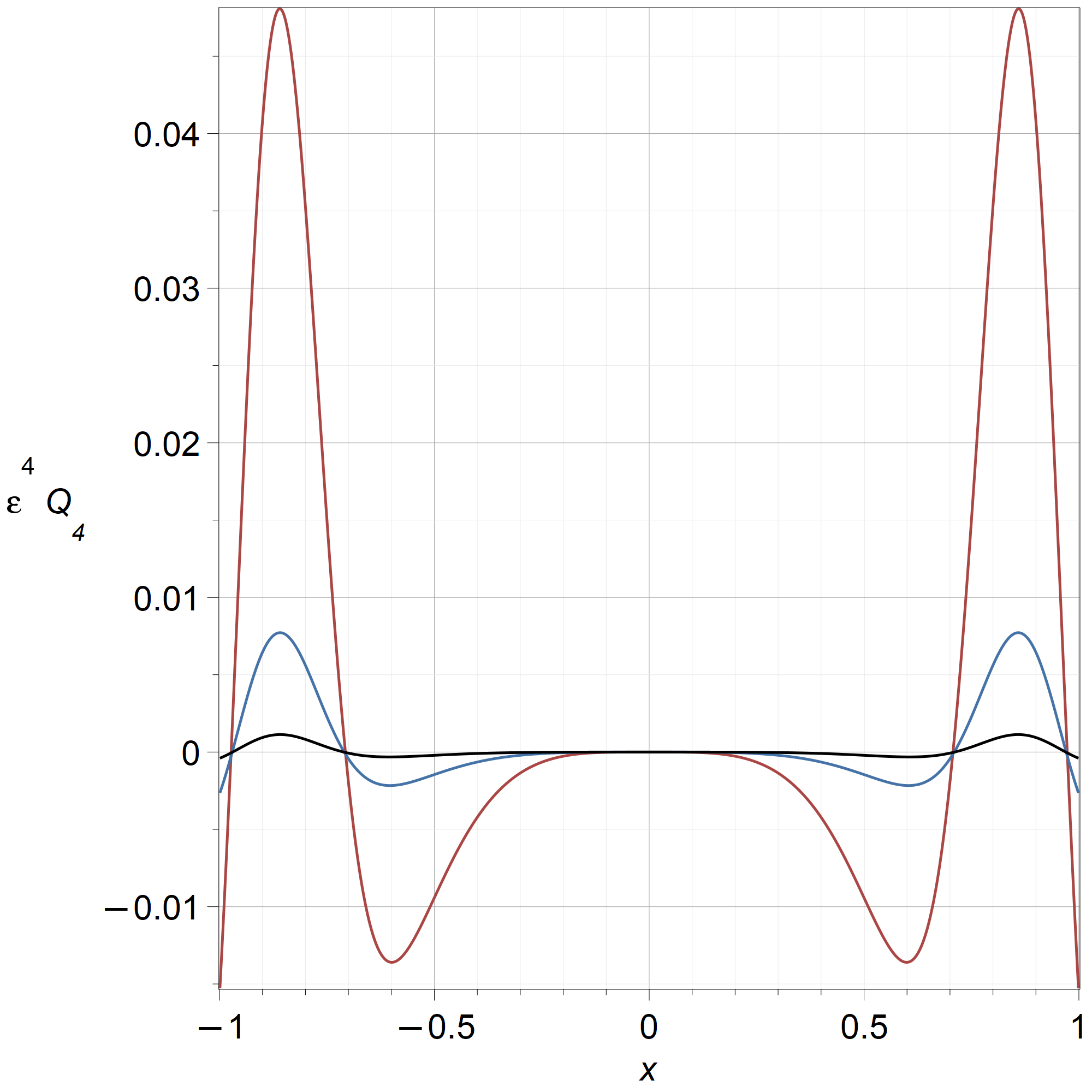}}
  \caption[Higher order perturbed potential]{(left) The perturbations
  $\e^2Q_2(x) = 3\e^2(x^2-2)/(4(1+x^2)^2)$ and $\e^4 Q_4(x) = \e^4{ (36 x^{4}-90 x^{2}+9)}/{(8 \left(x^{2}+1\right)^{5})} + O(\e^6)$ to the potential $1+x^2$ after one iteration of our basic perturbation method, starting from the WKB approximation. We show this for three different values of $\e$: $\e=1/5$ (red), $\e=1/8$ (blue), and $\e=1/13$ (black). One single iteration improves the backward error remarkably. (right) The same thing, but for the potential $Q(x)=1+x^8$. Here $Q_2 = 2 x^{6} \left(3 x^{8}-7\right)/(1+x^8)^2$. The expression for $Q_4$ is this time somewhat more complicated: to leading order, it is $Q_4 = 3 x^{4} \left(75 x^{24}-749 x^{16}+581 x^{8}-35\right)\e^4/Q(x)^5 + O(\e^6)$. The full residuals are plotted here for all figures, not series approximation to them.}\label{fig:PerturbedPotential4th}
\end{figure}

\end{example}

Applying this idea in general (for symbolic $Q$) gets us the following.
\begin{theorem}[Corollary to the WKB Backward Theorem]\label{thm:CorollaryWKBbackward} \index{WKB Backward Theorem!corollary}
As for Theorem~\ref{thm:BackwardWKB}, consider equation \eqref{eq:WKBSchroedinger}. Suppose $Q(x)$ is \textsl{four} times continuously differentiable, and is not zero on $a < x < b$. Let $\widehat{Q}(x) = Q(x) \textcolor{red}{-} \e^2 Q_2(x)$ where $Q_2(x)$ is as in Theorem~\ref{thm:BackwardWKB}.  If neither $Q(x)=0$ nor $\widehat{Q}(x) = 0$, then the WKB solution $y = c_1 \widehat{Q}^{-1/4}(x) \exp( \widehat{S}_0/\e ) + c_2 \widehat{Q}^{-1/4}(x) \exp( \widehat{S}_0/\e )$ to \textsl{this} problem, where
$$
\widehat{S}_0 = \int_a^x \sqrt{\widehat{Q}(\xi)}\,d\xi\>,
$$
has residual $r(x;\e) = \e^2 y'' - Q(x) y = \e^4 Q_4(x;\e)$ \textbf{in the original problem using $Q$ and not $\widehat{Q}$}, with
\begin{equation}\label{eq:modifiedResidualWKB2}
  r(x,\e)=\e^4 Q_4(x;\e) = \e^4\frac{K_1 + \e^2 K_2}{4096Q^6(x) \widehat{Q}^2(x)}\>,
\end{equation}
where
{\footnotesize
\begin{align}\label{eq:WKB2modrlabels}
  K_1 &= 8 Q \! \left(x \right)^{6} \left(Q^{iv}(x)\right)-56 Q \! \left(x \right)^{5} \left(Q'(x)\right) \left(Q'''(x)\right)-36 \left(Q''(x)\right)^{2} Q \! \left(x \right)^{5}\nonumber\\
  &\qquad{}+216 \left(Q''(x)\right) Q \! \left(x \right)^{4} \left(Q'(x)\right)^{2}-135 Q \! \left(x \right)^{3} \left(Q'(x)\right)^{4}
 \\
  K_2 &= -288 \left(Q''(x)\right)^{3} Q \! \left(x \right)^{3}+468 \left(Q''(x)\right)^{2} Q \! \left(x \right)^{2} \left(Q'(x)\right)^{2}+64 \left(Q''(x)\right) Q \! \left(x \right)^{4} \left(Q^{iv}(x)\right)\nonumber\\
  &\qquad{}+272 \left(Q''(x)\right) Q \! \left(x \right)^{3} \left(Q'(x)\right) \left(Q'''(x)\right)-540 \left(Q''(x)\right) Q \! \left(x \right) \left(Q'(x)\right)^{4}\nonumber\\
  &\qquad{}-80 Q \! \left(x \right)^{4} \left(Q'''(x)\right)^{2}-80 Q \! \left(x \right)^{3} \left(Q'(x)\right)^{2} \left(Q^{iv}(x)\right)-40 Q \! \left(x \right)^{2} \left(Q'(x)\right)^{3} \left(Q'''(x)\right)\nonumber\\
  &\qquad{}+225 \left(Q'(x)\right)^{6}\>.
\end{align}
}
Zeros of $\widehat{Q}(x)$ which are not also zeros of $Q(x)$ are termed \textbf{spurious turning points}.\index{turning point!spurious} We will see a method to remove them, so let us ignore them for the moment. 
Therefore under the hypotheses of this corollary this additional iteration provides the exact solution to the differential equation $\e^2 y'' = (Q(x) + \e^4 Q_4(x;\e))y$.
\end{theorem}
\begin{proof}
This time we use Maple to prove this theorem. First, we define the following procedure:
\begin{lstlisting}[caption=A Maple Procedure for WKB for Schr\"odinger-type equations]
WKB2Q := proc(Q::operator, Qorig::operator, x, eps, {a := 0})
  local xi, residual1, residual2, S, y1, y2;
  S := int(sqrt(Q(xi)), xi = a .. x);
  y1 := exp(S/eps)/Q(x)^(1/4);
  residual1 := simplify(eps^2*diff(y1, x, x)/y1 - Qorig(x));
  y2 := exp(-S/eps)/Q(x)^(1/4);
  residual2 := simplify(eps^2*diff(y2, x, x)/y2 - Qorig(x));
  return [y1,y2], [residual1,residual2];
end proc:
\end{lstlisting}
We then execute this procedure by the following commands:
\begin{lstlisting}
macro( ep=varepsilon );
(secondordersol,residual0) := WKB2Q(x -> Q(x), x -> Q(x), x, ep);
residual0[1]-residual0[2]; # yields 0
Q1 := unapply( Q(x) - residual0[1], x );
(fourthordersol, residual1) := WKB2Q(Q1, x -> Q(x), x, ep);
residual1[1]-residual1[2]; # yields zero again.
\end{lstlisting}
The fact that the residual for $y_1$ is the same as the residual for $y_2$ means that it will be the same for a linear combination of the two, and therefore able to be pulled in to the potential.

The command
\begin{lstlisting}
denom( residual1[1] );
\end{lstlisting}
yields
\begin{equation}
16 \left(5 \left(\frac{d}{d x}Q \! \left(x \right)\right)^{2} \varepsilon^{2}-4 \left(\frac{d^{2}}{d x^{2}}Q \! \left(x \right)\right) \varepsilon^{2} Q \! \left(x \right)-16 Q \! \left(x \right)^{3}\right)^{2} Q \! \left(x \right)^{2} \>,
\end{equation}
which shows that the denominator of the residual contains the factor $Q_1$ as well as $Q$.  The command
\begin{lstlisting}
collect(numer(residual1[1]), ep, m -> LargeExpressions:-Veil[K](m))
\end{lstlisting}
yields
\begin{equation}
-\varepsilon^{6} K_{1}-32 \varepsilon^{4} K_{2}\>,
\end{equation}
giving the equations of the theorem (apart from numbering).

To recognize the denominator in the residual, issue the following command:
\begin{lstlisting}
spurious := denom(residual1[1]):
normal( spurious/Q1(x)^2 );
\end{lstlisting}
That last command gives $4096 Q^6(x)$.
For completeness, here are the commands to reveal the contents of $K_1$ and $K_2$:
\begin{lstlisting}
for k to 2 do
    K[k] = LargeExpressions:-Unveil[K](K[k]);
end do;
\end{lstlisting}
The numbering of the expressions from the output of that session (not shown, but you can execute the commands yourself) is different but equivalent to that of the theorem.
\end{proof}

This corollary shows that a single iteration improves the backward error from $O(\e^2)$ to $O(\e^4)$, provided that all the steps can be carried out and that there are no turning points or spurious turning points in the region of interest.

\begin{remark}
As is well-known, turning points where $Q(x)=0$ have to be treated specially.
Notice that the difficulty is visible in the solution itself, which contains a factor $Q(x)^{-1/4}$ that goes to infinity at zeros of $Q(x)$, but it's \textsl{more} visible in the residual, which goes to infinity like $Q(x)^{-2}$, even relative to the growing solution.  
\end{remark}

Spurious turning points\index{turning points!spurious} are places where the denominator of $\widehat{Q}(x)$ is zero but $Q(x)$ is not zero.  An example is shown in the exercises in~\cite{CorlessFillion2025}. Approximating one problem $\e^2 y'' = Q(x)y$ by another which has spurious turning points does not seem useful; in mitigation we point out that these spurious zeros are likely to exist only for very large $\e$ and thus unlikely to be important.  

They can be removed, however, by the following trick\index{turning point!removing spurious}.  Put
\begin{align}\label{eq:RemoveSpurious}
  \widehat{Q}^{-1/4} &= (Q - \e^2 Q_2)^{-1/4} \nonumber\\
                     &= Q^{-1/4}\left( 1 - \e^2 \frac{Q_2}{Q} \right)^{-1/4}\nonumber\\
                     &= Q^{-1/4}e^{\ln\left( 1 - \e^2 \frac{Q_2}{Q} \right)^{-1/4}} \\
                     &= Q^{-1/4} e^{\e^2 Q_2/(4Q) + O(\e^4)}\>.
\end{align}
This trick of taking the logarithm of a series and then exponentiating\index{series!taking logarithm of} it again is something that is seen in the Renormalization Group Method for perturbation~\cite{Kirkinis(2012)}.  The solution still has an $O(\e^4)$ residual, although we will see later that we may lose the strong backward error result: it might not be true that this renormalized solution is the exact solution of a problem of the same type as the original.

The goal of the iteration is to approximate a function $\widetilde{Q}$ that satisfies
\begin{equation}
    \widetilde{Q}(x) + \e^2 \left( 5\left(\frac{\widetilde{Q}'}{4\widetilde{Q}}\right)^2 - \frac{\widetilde{Q}''}{4\widetilde{Q}}\right) = Q(x)\>,
\end{equation}
and is ``close'' to the original and desired $Q(x)$. Starting with this $\widetilde{Q}$ would mean that the WKB algorithm would give us the exact solution to the original problem.

Although this looks like a differential equation for $\widetilde{Q}$, we would be happy to find any solution at all, so long as it's close to $Q$. To find a solution, we have to solve that equation for $\widetilde{Q}$, given $Q$.  We use functional iteration:
\begin{equation}
    Q_{n+1} = Q(x) - \e^2 \left( 5\left(\frac{{Q_n'}}{4{Q_n}}\right)^2 - \frac{{Q_n}''}{4{Q_n}}\right)
\end{equation}
with $Q_0(x) = Q(x)$.  This iteration generates one more term correct in the (even) power series for $\widetilde{Q}$ with each pass.  

\begin{remark}
The repeated differentiation in that iteration might be a concern. As is well-known, differentiation can introduce unwanted growth in the series coefficients.
\end{remark}

This gives us \textbf{The Iterative WKB Algorithm}\index{WKB!iterative version}, which we present in algorithm~\ref{alg:IWKB}.  This iteration has backward error that formally decreases with each iteration: $O(\e^2)$ at $n=0$, $O(\e^4)$ at $n=1$, $O(\e^6)$ at $n=2$, and so on. In practice, only the first few iterations are likely to be useful, because of the rapidly increasing complexity of $Q_n$, which makes more difficult the step labelled ``WKB'' in this algorithm (which just means compute the final $\int_{a}^{x} \sqrt{Q_N(\xi)}\,d\xi$ and (optionally) remove the spurious turning points in $Q_{N}^{-1/4}$ by the use of equation~\eqref{eq:RemoveSpurious}).  As previously stated, having more and higher derivatives of $Q$ becomes involved, which is also a concern.
\begin{algorithm}
\label{alg:IterativeWKB}
\begin{algorithmic}
\Procedure{IWKB}{$\mathcal{Q}$, $\e$, $N$ }\Comment{$\e^2 y'' =Q(x)y$  }
\State $n \gets 0$
\State $Q_n \gets \mathcal{Q}$ 
\While{$n < N$}
  \State $Q_{n+1} \gets \mathcal{Q} - \e^{2}\left(5(Q_n'/(4Q_n))^2 - (Q_{n}''/(4Q_n)\right)$
  \State $n \gets n+1$
\EndWhile
\State $y \gets \mathrm{WKB}(Q_N)$ \Comment{Use WKB on $\e^2y'' = Q_N y$ final step} 
\State \textbf{return}~$y$\Comment{solution $O(\e^{2N+2})$ backward error}
\EndProcedure
\end{algorithmic}
\caption{The iterated WKB algorithm.\label{alg:IWKB}}
\end{algorithm}
\begin{remark}
Does this iteration converge?  This does not seem to be a useful question. The normal $N$ to use for this algorithm is just $N=1$, and if that helps \textsl{a lot}, then you won't need $N=2$. If it doesn't help very much but does help a little, then one might try $N=2$. However, given how much more complicated the integrals get, the iteration is unlikely to be of much more use than the choice $N=1$ was, if one is doing symbolic integration. We will see in the next section a way to proceed to higher order with this iteration, however, by using Chebfun for the integration.
\end{remark}
\begin{example}
Consider $Q(x) = -(1+x^8)$, as in our first example.  Applying one step of our iterative process gives us the exact solution to $\e^2y'' = Q_1 y$ where
\begin{equation}
    Q_1 = -(1 + x^8) - \frac{3 \left(75 x^{24}-749 x^{16}+581 x^{8}-35\right) x^{4}}{\left(x^{8}+1\right)^{5}}\e^4 + O(\e^6)\>. 
\end{equation}
A plot of the rational function that is the coefficient of the $\e^4$ term shows that it is less than $35$ in magnitude. When $\e = 0.664$ approximately, there is a spurious turning point; but that $\e$ is so large that it's unlikely that the spurious turning point will be important.
\end{example}

\subsection{Comparing the iterative algorithm with standard WKB}
The standard WKB approximation can be carried out to whatever order is desired.  One chooses $n>1$ and posits $\mathcal{S} = S_0/\delta + S_1 + \delta S_2 + \cdots + \delta^{n-1} S_n$ and sets the first $n+1$ terms in the residual to be zero.  This gives well-known recurrence relations for the $S_k$ in terms of the previous $S_j$.  See for example~\cite{Bender(1978)} for a list.  The two approaches are strongly related, of course. We have, for instance, 
\begin{equation}
2 \left(\frac{d}{d x}S_{0} \! \left(x \right)\right) \left(\frac{d}{d x}S_{2} \! \left(x \right)\right)+\left(\frac{d}{d x}S_{1} \! \left(x \right)\right)^{2}+\frac{d^{2}}{d x^{2}}S_{1} \! \left(x \right) = 0
\end{equation}
to define $S_2(x)$.  Using what we know of $S_0$ and $S_1$ this becomes
\begin{equation}
2 \sqrt{Q \! \left(x \right)}\, \left(\frac{d}{d x}S_{2} \! \left(x \right)\right)+\frac{5 \left(\frac{d}{d x}Q \! \left(x \right)\right)^{2}}{16 Q \! \left(x \right)^{2}}-\frac{\frac{d^{2}}{d x^{2}}Q \! \left(x \right)}{4 Q \! \left(x \right)} = 0
\end{equation}
and we recognize our $Q_2(x)$ on the right.  Setting this to zero (note that the choice of sign for $\sqrt{Q(x)}$ is needed here to make it consistent) gives a solution with residual $O(\e^3)$.  Here, we have to do another integration in order to find $S_2(x)$.

The next term gives
\begin{equation}
\frac{d}{d x}S_{3} \! \left(x \right) = 
-\frac{15 \left(\frac{d}{d x}Q \! \left(x \right)\right)^{3}}{64 Q \! \left(x \right)^{4}}+\frac{9 \left(\frac{d^{2}}{d x^{2}}Q \! \left(x \right)\right) \left(\frac{d}{d x}Q \! \left(x \right)\right)}{32 Q \! \left(x \right)^{3}}-\frac{\frac{d^{3}}{d x^{3}}Q \! \left(x \right)}{16 Q \! \left(x \right)^{2}}
\end{equation}
and while the sign of the square root of $Q$ is immaterial this time, we still seem to have to do another integration. But this time the integral can be done symbolically, for any $Q(x)$:
\begin{equation}
    S_3(x) = \frac{5 \left(\frac{d}{d x}Q \! \left(x \right)\right)^{2}}{64 Q \! \left(x \right)^{3}}-\frac{\frac{d^{2}}{d x^{2}}Q \! \left(x \right)}{16 Q \! \left(x \right)^{2}}\>.
\end{equation}
Again we see our $Q_2(x)$, this time divided by $4Q$. Compare this to equation~\eqref{eq:RemoveSpurious}.

This suggests that both methods are comparable in effort.  But, since we have already written code for the basic WKB computation, and it can be re-used to carry out the iterative method, it seems operationally simpler to use the iterative scheme---except if we have to deal with spurious turning points, which the standard method does not produce.  If we \textsl{do} wish to remove the spurious turning points, then the procedure outlined previously, together with the perturbative nature of doing the integral of $\sqrt{Q(x) - \e^2 Q_2(x)}$, turns the iterative scheme into \textsl{exactly} the standard approach. So the new method looks the same as the old method, so far.

But it's not identical, and in particular the property that the relative residual of $y_1 = Q^{-1/4}\exp(S_0/\e - \e S_2)$ and the relative residual of $y_2 = Q^{-1/4}\exp(-S_0/\e + \e S_2)$ be the same does \textsl{not} hold: In the first case we get an absolute residual $\e^3 Q_3 y_1 + O(\e^4)$ but in the second we get $-\e^3 Q_3 y_2 + O(\e^4)$.  This surprised us. The difference is caused by the term below, which contains a square root of $Q$:
\begin{equation}
    Q_3 = \frac{4 Q \! \left(x \right)^{2} \left(\frac{d^{3}}{d x^{3}}Q \! \left(x \right)\right)-18 Q \! \left(x \right) \left(\frac{d^{2}}{d x^{2}}Q \! \left(x \right)\right) \left(\frac{d}{d x}Q \! \left(x \right)\right)+15 \left(\frac{d}{d x}Q \! \left(x \right)\right)^{3}}{32 Q \! \left(x \right)^{\frac{7}{2}}}\>.
\end{equation}
Because this is nonzero and has opposite signs in the residual for $y_1$ and in $y_2$, the relative residual of $c_1y_1 + c_2y_2$ is not proportional to $c_1y_1 + c_2y_2$.

Taking one more term, with $y_1 = Q^{-1/4}\exp(S_0/\e - \e S_2 + \e^2 S_3)$ and  $y_2 = Q^{-1/4}\exp(-S_0/\e + \e S_2 + \e^2 S_3)$ we \textsl{almost} recover equality: The relative residual of $y_1$ is now $\e^4 Q_4 + \e^5 Q_5 + \cdots$ while the relative residual of $y_2$ is $\e^4 Q_4 - \e^5 Q_5 + \cdots$, so the difference does not appear until the $O(\e^5)$ term in the residual:
\begin{equation}
    Q_5 = -\frac{K_5}{1024 Q \! \left(x \right)^{\frac{13}{2}}}
\end{equation}
where
\begin{align}
    K_5 =&\> 72 \left(\frac{d^{2}}{d x^{2}}Q \! \left(x \right)\right)^{2} Q \! \left(x \right)^{2} \left(\frac{d}{d x}Q \! \left(x \right)\right)-150 \left(\frac{d^{2}}{d x^{2}}Q \! \left(x \right)\right) Q \! \left(x \right) \left(\frac{d}{d x}Q \! \left(x \right)\right)^{3}\nonumber\\
    &{}-16 \left(\frac{d^{2}}{d x^{2}}Q \! \left(x \right)\right) Q \! \left(x \right)^{3} \left(\frac{d^{3}}{d x^{3}}Q \! \left(x \right)\right)+75 \left(\frac{d}{d x}Q \! \left(x \right)\right)^{5}\nonumber\\
    &\qquad{}+20 \left(\frac{d}{d x}Q \! \left(x \right)\right)^{2} \left(\frac{d^{3}}{d x^{3}}Q \! \left(x \right)\right) Q \! \left(x \right)^{2}
\>.
\end{align}

This means that we cannot, in general, interpret the standard WKB solutions with $n>1$ as being the exact solution to a nearby problem of the same kind. The third order truncation gives the solution to $\e^2 y'' = Q y + O(\e^3)$ but the $O(\e^3)$ term cannot be brought in to the $Q$ term.  The fourth order truncation (unlike the iterative method) gives the exact solution to $\e^2 y'' = (Q + \e^4 Q_4)y + O(\e^5)$ where the $O(\e^5)$ term cannot be brought in to the $Q$ term.  This may be an advantage to the iterative method, or it may not be significant at all, depending on the problem at hand. 

We point out that the iterative scheme has a residual that is even in the variable~$\e$.  This means that there are no odd powers of~$\e$ in the residual.

Another difference worth mentioning concerns the scope of the methods. The standard WKB method is applicable to many equations that are not explicitly in the form of equation~\eqref{eq:WKBSchroedinger}, but can be reduced to that form. For instance, equations of the form $y'' + a(x)y' + b(x)y = 0$ can be transformed by use of the Sturm transformation \cite[p.~12]{Smith1985}\index{Sturm transformation}, by setting 
\begin{equation}\label{eq:SturmTransformation}
  v(x) = \exp\left( \frac12 \int_{x_0}^x a(\xi)\,d\xi \right) y(x)\>.
\end{equation}
This transforms the equation to
\begin{equation}\label{eq:SturmTransformed}
  v''(x) + c(x) v(x) = 0 \>,
\end{equation}
where
\begin{equation}\label{eq:SturmTc}
c(x) = \frac14\left( 4b(x) - a^2(x) - 2a'(x) \right)\>,
\end{equation}
and where the initial values transform to $v(x_0) = y(x_0)$ and $v'(x_0) = y'(x_0) + a(x_0)y(x_0)/2$. However, there are cases where this transform will cause difficulties with the backward error discussed here. See for instance exercise { 8.7.13} in \cite{CorlessFillion2025} for a case where adding a first derivative term adds complications.  Moreover, Murdock points out other difficulties with this transformation in~\cite{Murdock1999}.

\section{A new hybrid method\label{sec:hybrid}}
The bottleneck for symbolic computing with the WKB method is the computation of $\int \sqrt{Q(\xi)}\,d\xi$. Even if symbolic integration succeeds, the result may require more careful handling than one thinks.  For example, consider the problem
\begin{equation}
    \e^2 y'' = \cosh(x) y
\end{equation}
subject to the boundary conditions $y(-1)=1$ and $y(1)=1$. Maple gives an answer that can be ``simplified'' to the following unsatisfactory expression:
\begin{equation}
    \int \sqrt{\cosh \xi}\,d\xi = -\frac{\sqrt{2}\, \sqrt{1-\cosh \! \left(\xi \right)}\, \sqrt{-\cosh \! \left(\xi \right)}\, E\! \left(\cosh \! \left(\frac{\xi}{2}\right), \sqrt{2}\right) \mathrm{csch}\! \left(\frac{\xi}{2}\right)}{\sqrt{\cosh \! \left(\xi \right)}}\>.
\end{equation}
If instead we evaluate a definite integral where we let Maple know that $x > 0$ via
\begin{lstlisting}
int(sqrt(cosh(xi)), xi = 0 .. x) assuming x > 0;
\end{lstlisting}
then the result is much more palatable:
\begin{equation}\label{eq:ellipticPi}
    \int_0^x \sqrt{\cosh \xi}\,d\xi = \sqrt{2}\, \Pi \! \left(\sqrt{\frac{\cosh \! \left(x \right)-1}{\cosh \! \left(x \right)}}, 1, \frac{\sqrt{2}}{2}\right)\>.
\end{equation}
The WKB expressions involve the \lstinline{Elliptic} functions $E$ and $\Pi$, which are implemented quite well in Maple. Unfortunately, the expression in equation~\eqref{eq:ellipticPi}  is correct only for $x > 0$.  The integral is the negative of that if $x < 0$. Such ``discontinuities in expressions'' for a smooth function---the solution to the original differential equation is a generalized Mathieu function and is entire---are quite common, because the alphabet of special functions frequently has places where one must switch representations.  

Hence if we wish to preserve the advantages of the WKB method, while realizing that the method gives exact solutions to slightly different problems anyway, we might try to approximate $Q(x)$ to start with in a way that makes integration of $\sqrt{Q(x)}$ simple.  A natural idea is to use Chebfun~\cite{battles2004extension} to replace the awkward special functions with approximation by Chebyshev polynomials.  We can demonstrate the idea using Maple's facilities for Chebyshev expansion, which themselves are quite advanced (and indeed are among the oldest facilities in Maple, dating back to~\cite{geddes1993package}).

For this example, approximate $\sqrt{\cosh x}$ on $-1 \le x \le 1$ by a Chebyshev expansion by issuing the command
\begin{lstlisting}
f := numapprox[chebyshev]( sqrt( cosh(xi) ), xi=-1..1 );
\end{lstlisting}
which yields
\begin{align}
f(x) =&     1.12193597495626 T \! \left(0, \xi \right)+ 0.121038841819815 T \! \left(2, \xi \right)\nonumber\\
&{}- 0.0008283990372008 T \! \left(4, \xi \right)+ 0.0000649230623055 T \! \left(6, \xi \right)\nonumber\\
&{}- 3.577974919\times 10^{-6} T \! \left(8, \xi \right)+ 2.17976072\times 10^{-7} T \! \left(10, \xi \right)\nonumber \\
&{}- 1.406882\times 10^{-8} T \! \left(12, \xi \right)
+ 9.46440\times 10^{-10} T \! \left(14, \xi \right)- 6.5657\times 10^{-11} T \! \left(16, \xi \right)\nonumber\\
&{}+ 4.664\times 10^{-12} T \! \left(18, \xi \right)
- 3.38\times 10^{-13} T \! \left(20, \xi \right)+ 2.5\times 10^{-14} T \! \left(22, \xi \right)\nonumber\\
&{}- 2.\times 10^{-15} T \! \left(24, \xi \right)\>.
\end{align}
This approximation is constructed to be accurate to double precision on the stated interval, which admittedly was chosen to be convenient for unscaled Chebyshev polynomials.

The notation $T(n,\xi)$ is an older ``inert'' Maple notation for the Chebyshev polynomials, suitable for further manipulation such as integration using the rules $\int T_0(x)\,dx = T_1(x)$, $\int T_1(x)\,dx = (T_0(x)+T_2(x))/2$, and $\int T_k(x)\,dx = T_{k+1}(x)/(2(k+1)) - T_{k-1}(x)/(2(k-1))$.  

The integral of this expression is 
\begin{equation}
   F(x) =  \int_0^x f(\xi)\,d\xi = \sum_{k=0}^{11} c_{2k+1} T_{2k+1}(x)
\end{equation}
where we computed all twelve coefficients $c_{2k+1}$  but do not print them here, for space reasons; for instance $c_1 = 1.06141655404635$.  As in Chebfun, this expression differs from the true integral only by about rounding error levels (in this case less than $\snot{5}{-15}$).

Higher precision could be used if desired. \textsl{Lower} precision could be used if desired, because with (say) $\e = 1/100$ the residual is going to be $O(\e^2)$ or about $10^{-4}$ anyway, so working with Chebyshev polynomial expressions accurate to $10^{-15}$ seems like overkill. Working to six figures would give
\begin{align}
    f(x) &=  1.12194+ 0.121039 T_{2}\! \left(x \right)- 0.000828399 T_{4}\! \left(x \right) \\
    F(x) &= 1.06142 T_{1}\! \left(x \right)+ 0.0203112 T_{3}\! \left(x \right)\>.
\end{align}
In practice it's pretty convenient to leave the expressions accurate to double precision, because then they are easier to check.

Now the WKB-like expression
\begin{equation}
    y = c_1 f(x)^{-1/2} e^{F(x)/\e} + c_2 f(x)^{-1/2} e^{-F(x)/\e}
\end{equation}
(note the powers $-1/2$ and not $-1/4$ for $f(x)$)
gives \textsl{nearly} the exact solution to the perturbed problem discussed before, because $f(x)^2 = \widetilde{Q} \approx \cosh(x)$ up to rounding error. The residual $r(x) = \e^2 y'' - f^2(x) y$ is, as before, $-\e^2 (\widetilde{Q}''/(4\widetilde{Q}) - 5(\widetilde{Q}'/(4\widetilde{Q}))^2) y$, meaning that we have exactly solved a problem with a perturbed potential that is $O(\e^2)$ plus the rounding error difference to $\cosh(x)$.  

This time there are no alternative representations needed for different regions. 

For this example, the solution has two simple boundary layers of width $O(\e)$ at each end, and the graph is relatively uninteresting.  Indeed, standard numerical methods such as MATLAB's \lstinline{bvp4c}~\cite{kierzenka2008bvp} work quite well on this problem for reasonably sized $\e$. 

\begin{example}
Let us try a harder example, with $-\cosh(x)$ instead of $\cosh(x)$, which produces densely oscillatory behaviour as $\e \to 0$.  All of the above Chebyshev polynomials can be re-used.  The only difference is that the exponentials are now multiplied by $i$:
\begin{equation}\label{eq:ChebyshevWKB}
    y = c_1 f(x)^{-1/2} e^{iF(x)/\e} + c_2 f(x)^{-1/2} e^{-iF(x)/\e}
\end{equation}

The residuals are the same as before.  We identify $c_1$ and $c_2$ by solving linear equations, to get $c_1 = c_2 \approx 0.743437$, when $\e = 1/21$ (a convenient value for plotting).  See figure~\ref{fig:chebyWKB}. When $\e=1/100$ the oscillations are more dense but the solution process is carried out without difficulty.

\begin{figure}
    \centering
    \includegraphics[width=0.7\textwidth]{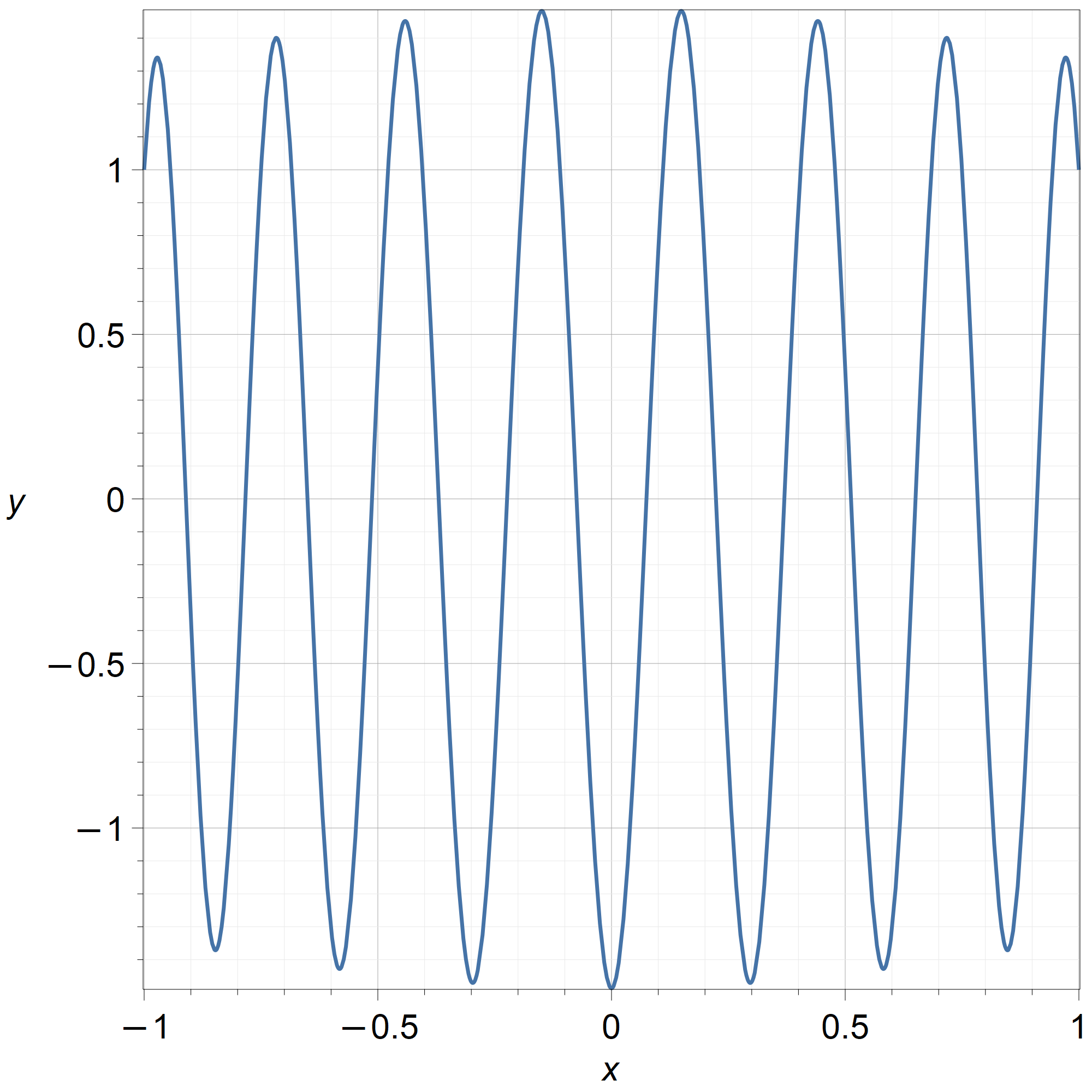}
    \caption{The Chebyshev--WKB solution to $\e^2 y'' + \cosh(x) y =0$ subject to $y(-1) = y(1) = 1$, when $\e=1/21$.}
    \label{fig:chebyWKB}
\end{figure}

\end{example}
For both those examples, the exact reference solution to the original equation is available in terms of solutions to the Mathieu equations.  Also, the WKB bottleneck integral can be carried out explicitly in terms of elliptic functions, which actually work well in Maple (if one is careful to use a piecewise representation for $x>0$ and for $x \le 0$).  The real strength of the Chebyshev approximation method would be shown by an example where none of those things were true.

\begin{example}
Suppose for instance that the potential is $2$ at $x=\pm 1$ and has zero derivative there, and suppose that the potential is $1$ at $x=\pm 1/2$ and again has zero derivative there.  This gives a kind of ``double well'' potential.  See figure~\ref{fig:WKBHardExample}.

\begin{figure}
\centering
\subfigure[Potential  \label{fig:WKBHardExample}]{\includegraphics[width=.48\textwidth]{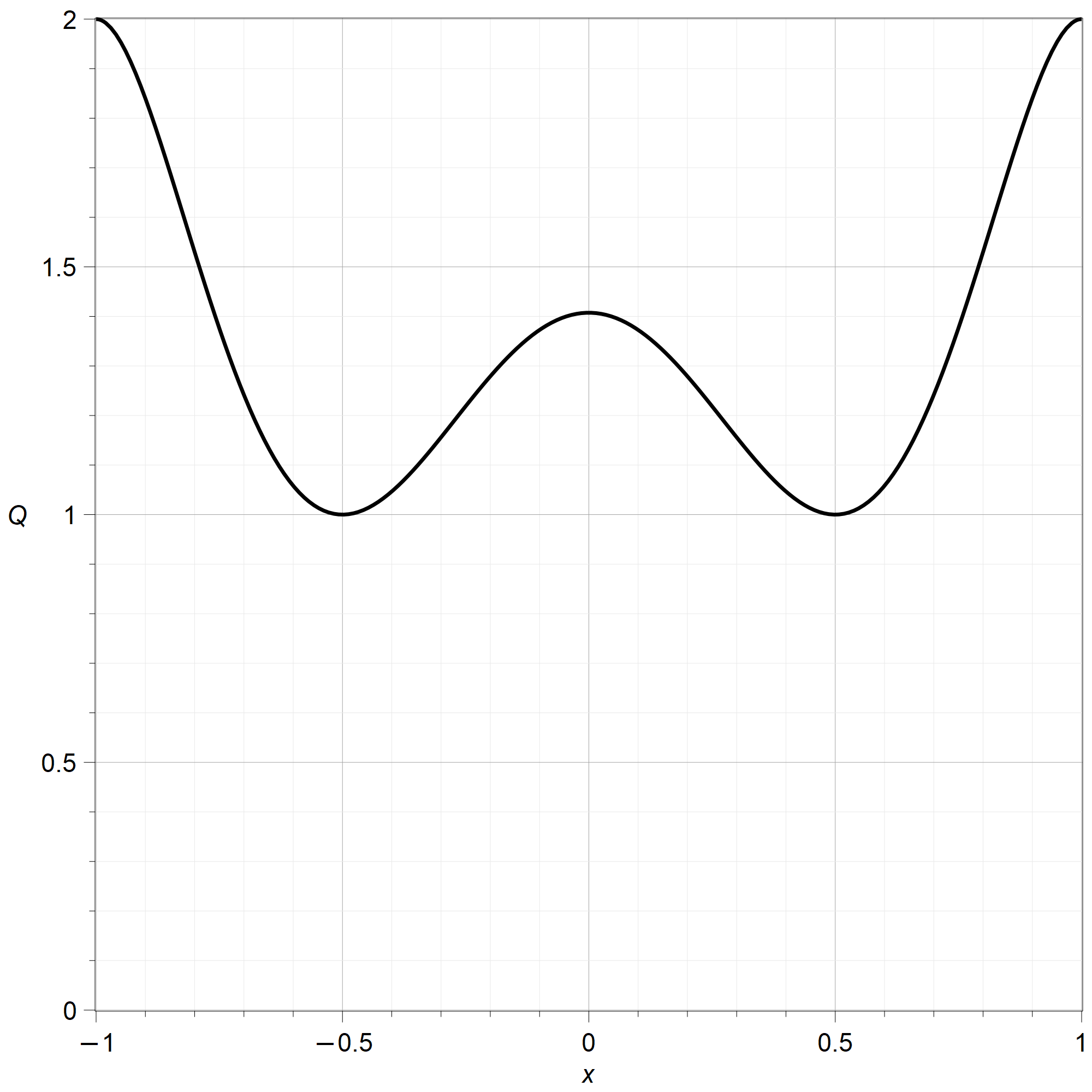}}
\subfigure[Residual \label{fig:HardResidual}]{\includegraphics[width=.48\textwidth]{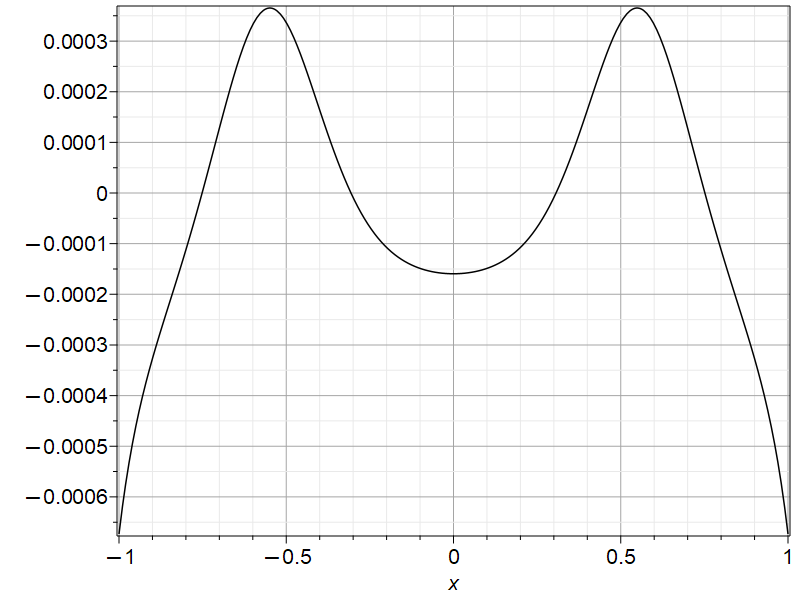}}
\caption[A hard example]{(left) A potential $Q(x)$ that is $2$ at the endpoints and $1$ at $\pm 1/2$, with zero derivatives at all four places.  (right) The perturbation $\e^2 Q_2$ to the potential when $\e = 1/89$.  If this were added to the potential on the left, it would not be visibly different. }
\end{figure}

The smallest degree polynomial that fits this data is $Q(x) = (38-96 x^{2}+240 x^{4}-128 x^{6})/27$.  Attempting the WKB procedure on this gives integrals for $\sqrt{Q}$ that Maple does not know how to express except as quadratures.  If we are willing to evaluate the integrals numerically, the WKB procedure still works, albeit slowly.

But if we approximate $\sqrt{Q}$ by a sum of Chebyshev polynomials, we get
\begin{equation}
    \sqrt{Q(\xi)} = \sum_{k=0}^{37} c_{2k} T_{2k}(\xi)
\end{equation}
which integrates easily to get
\begin{equation}
    \int_0^x \sqrt{Q(\xi)}\,d\xi = \sum_{k=0}^{37} C_{2k+1} T_{2k+1}(x)\>.
\end{equation}
That's quite a few coefficients, but the computer doesn't mind. We take $y(0)=1$ and $y'(0)=0$ as initial conditions. When $\e = 1/8$, the residual is less than 5\% of the value of $Q$ across the interval (whereas $y$ is bigger than $10^4$ near $x = \pm 1$).

If instead we use boundary conditions $y(-1) = y(1) = 1$, then the WKB solution varies exponentially over 40 orders of magnitude for $\e=1/89$ down to about $10^{-42}$ at the origin.  For this $\e$, the perturbation $\e^2 Q_2$ to the potential has magnitude less than $0.0006$.  See figure~\ref{fig:HardResidual}.

This suggests that Chebyshev approximation of the potential can be a useful hybrid perturbation technique.
\end{example}

Finally, we try it in Chebfun proper. We adapt one of the example integrals from the web page, namely \href{https://www.chebfun.org/examples/quad/SymbolicNumeric.html}{the Symbolic-Numeric example}, where Nick Trefethen shows that Chebfun quadrature can beat symbolic integration.  We chose the potential $Q(x) = -1-|x|$ with initial conditions to get a cosine WKB formula. We have also done the difficult example $Q(x) = -128x^6/27+80x^4/9-32x^2/9+38/27$ (not shown here).

The code is below.  The results are shown in figure~\ref{fig:chebfunex}.
\begin{figure}
\centering
\subfigure[Solution for $\e=1/233$  \label{fig:WKBCheb}]{ \includegraphics[width=0.45\textwidth]{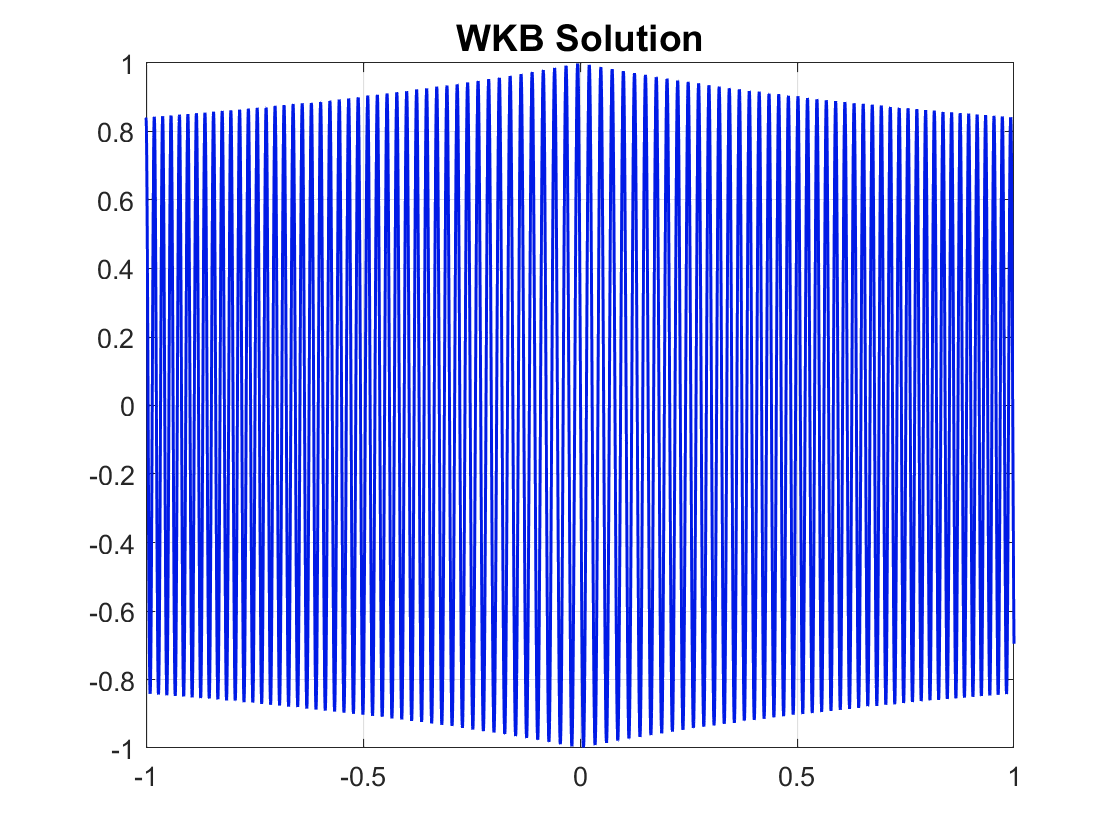}
 }
\subfigure[Residual for $\e=1/233$\label{fig:Chebres}]{\includegraphics[width=.48\textwidth]{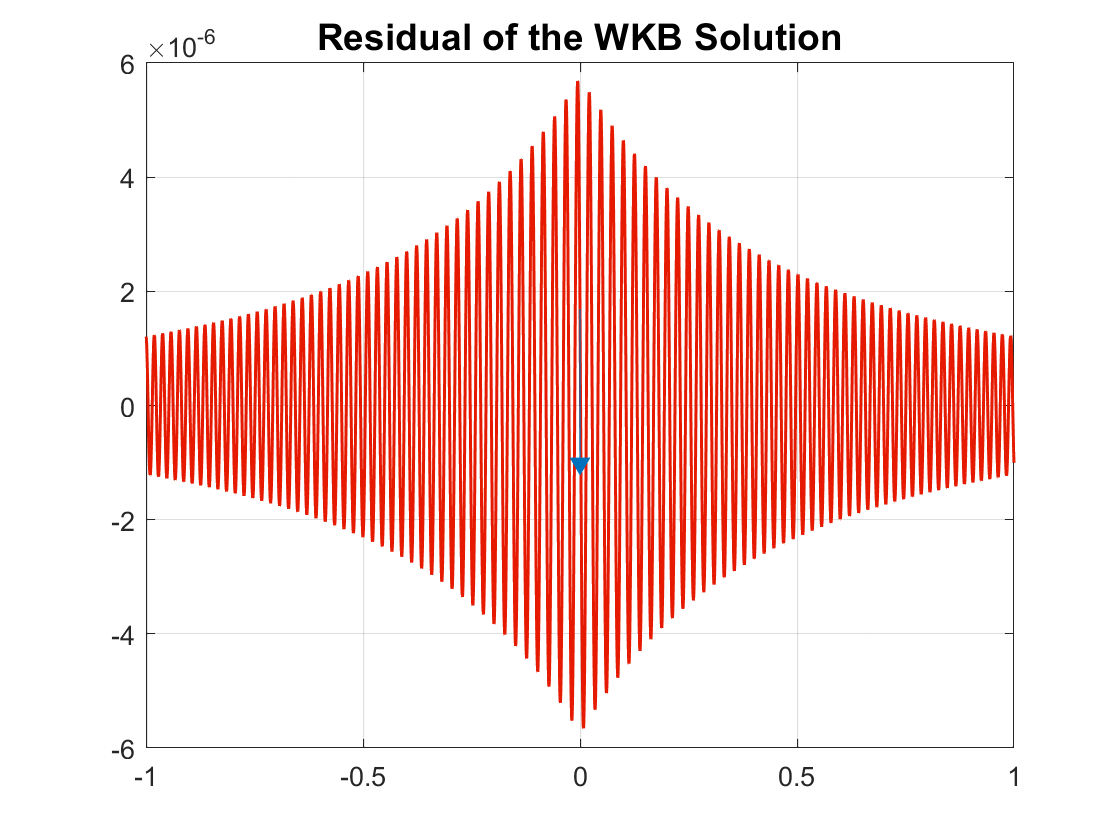}}
\caption[Chebfun and WKB]{(left) The WKB solution by Chebfun to $\e^2y'' + (1+|x|)y = 0$ for $\e=1/233$.  (right) The absolute residual in that solution, which is uniformly small. The vertical scale on the left figure is from $-1$ to $1$ while that on the right is from $-\snot{6}{-6}$ to $\snot{6}{-6}$. \label{fig:chebfunex} }
\end{figure}

\begin{lstlisting}
LW = 'LineWidth'; CO = 'Color'; FS = 'FontSize';
Q = chebfun(@(x) 1+abs(x), 'splitting', 'on')
f = sqrt(Q); # Q really -Q
fi = cumsum(f)
ep = 1.0/233.0;
y = f^(-0.5).*cos( fi/ep )
figure(1),plot(y,LW,1.2,CO,[0 .1 .9]), grid on
title('WKB Solution',FS,14)
res = ep^2*diff(diff(y)) + Q*y;
figure(2),plot(res,LW,1.2,CO,[0.9 .1 0]), grid on
title('Residual of the WKB Solution',FS,14)
\end{lstlisting}

We now try the iterated WKB using Chebfun, for $Q(x) = \cosh(x)$.  Since $Q$ is analytic, the higher derivatives in the error terms are not a problem.
The code is below.  The results are shown in figure~\ref{fig:iteratedchebfunex}.
\begin{figure}
\centering
\subfigure[Forward error \label{fig:Forward}]{ \includegraphics[width=0.45\textwidth]{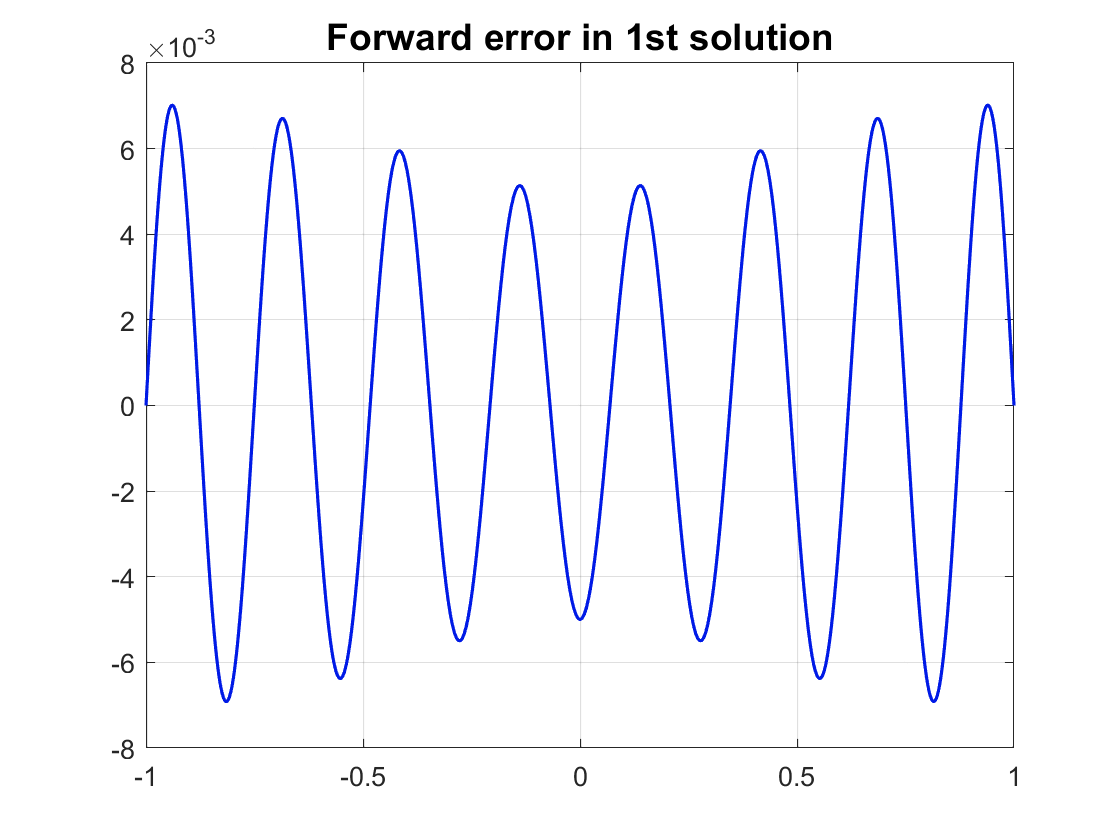}}
\subfigure[Residual after one iteration\label{fig:iteratedChebres}]{\includegraphics[width=.48\textwidth]{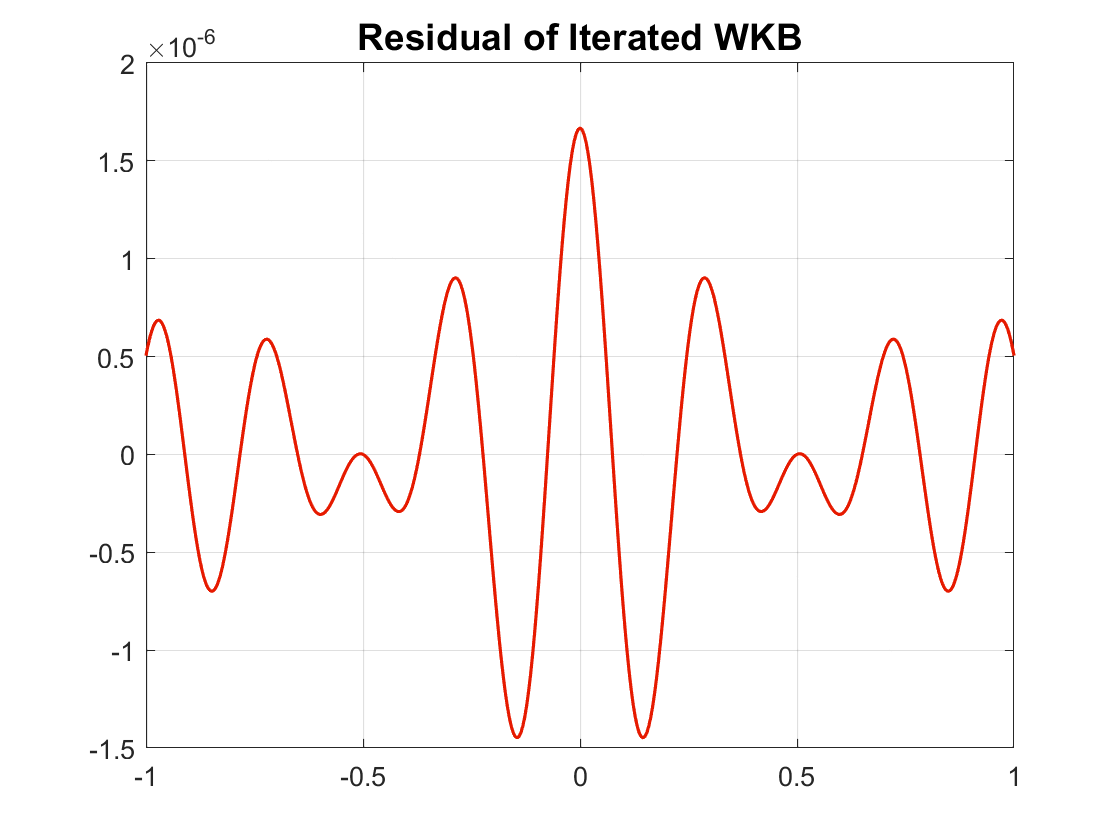}}
\caption[Chebfun and WKB]{(left) The forward difference $y-y_4$ between the WKB solution $y$ by Chebfun to $\e^2y'' + \cosh(x) y = 0$, with $y(-1)=y(1)=1$, for $\e=1/21$ and the iterated WKB solution $y_4$ (obtained with one iteration). This difference estimates the true forward error in the less accurate solution, which has the $O(\e^2)$ residual. We see that the error is plausibly $O(\e)$, although this is difficult to tell, even with runs with different $\e$.  (right) The absolute residual in the iterated solution, which is uniformly small (the vertical scale is from $-\snot{1.5}{-6}$ to $\snot{2}{-6}$) and apparently $O(\e^4)$ as it should be. \label{fig:iteratedchebfunex} }
\end{figure}

\begin{lstlisting}
clf
LW = 'LineWidth'; CO = 'Color'; FS = 'FontSize';
a = -1
b = 1
ya = 1
yb = 2
Q = chebfun( @(x) cosh(x), [a,b] )
f = sqrt(Q); # Q really -Q
fi = cumsum(f)
ep = 1.0/10.0;
y1 = f^(-0.5).*cos( fi/ep )
y2 = f^(-0.5).*sin( fi/ep )
A = [y1(a) y2(a)
     y1(b) y2(b)]
c = A\[ya;yb];
y = c(1)*y1 + c(2)*y2;
figure(1),plot(y,LW,1.2,CO,[0 .1 .9]), grid on
title('WKB Solution',FS,14)
res = ep^2*diff(diff(y)) + Q*y;
figure(2),plot(res,LW,1.2,CO,[0.9 .1 0]), grid on
title('Residual of the WKB Solution',FS,14)
% Now iterate
Qt = Q;
N = 8;
for i=1:N
    # Q really -Q
    Qt = Q + ep^2*( 5*(diff(Qt)/(4*Qt))^2 - (diff(diff(Qt))/(4*Qt)))
end;
f1 = sqrt(Qt); 
f1i = cumsum(f1)
y11 = f1^(-0.5).*cos( f1i/ep )
y12 = f1^(-0.5).*sin( f1i/ep )
A = [y11(a) y12(a)
     y11(b) y12(b)]
c1 = A\[ya;yb];
y1 = c1(1)*y11 + c1(2)*y12;
figure(3),plot(y1,LW,1.2,CO,[0 .1 .9]), grid on
title('Iterated WKB Solution',FS,14)
res1  = ep^2*diff(diff(y1)) + Q*y1;
figure(4),plot(res1,LW,1.2,CO,[0.9 .1 0]), grid on
title('Residual of the iterated WKB Solution',FS,14)
\end{lstlisting}

When we did the ``difficult'' example $Q(x) = -128x^6/27+80x^4/9-32x^2/9+38/27$ Chebfun had no difficulty whatever.  For $\e=1/21$ the residual error was less than $\snot{5}{-4}$, and the iterated WKB had residual error less than $\snot{7}{-7}$.

The code as listed above has \textsl{eight} iterations of the IWKB method, which would be almost impossible with purely symbolic computation. Here, for $\e=1/10$, the backward error of $Q_0$ has maximum magnitude about $\snot{2.5}{-2}$, while the residual after eight iterations has maximum magnitude $\snot{6}{-10}$. This suggests that the derivatives are growing somewhat, so we are not achieving the full $O(\e^{18})$ accuracy, but we still see significant improvement. As we said before, we don't really care whether or not this iteration would converge if we iterated infinitely often.  Rounding errors would eventually accumulate and spoil the iteration anyway.  For this example the best we can do appears to be in the $N=8$ to $N=10$ range, all of which produce residuals about $\snot{5}{-10}$ or so. It's worse with $N=10$ than it is with $N=9$ and gets worse still as $N$ increases from there.  Whether this is divergence (this seems probable) or rounding error we do not know, and we have not investigated the question because knowing the answer would have no impact.

As a final example, we solve $\e^2 y'' = (1+(x-1/4)^4)y=0$ subject to $y(0)=1$ and $y(2) = 1$, making suitable modifications to the script above.  We chose $\e=1/161.5$. The residuals are plausibly $O(\e^2)$ and $O(\e^4)$ as expected, and the estimated forward error is about $\snot{1}{-3}$ which is again plausibly $O(\e)$.  See figure~\ref{fig:WKBCex2g}.

If instead we try to solve $\e^2 y'' = (1+|x-1/4|^{1/2})y$ with the same boundary conditions, the iterated WKB method fails with the very appropriate error message ``Delta functions at the same point cannot be multiplied''. The basic method, however, succeeds in computing an answer although the residual is apparently infinite at $x=1/4$.

\begin{figure}
    \centering
\subfigure[Solution for $\e=1/161.5$  \label{fig:WKBCex2}]{ \includegraphics[width=0.45\textwidth]{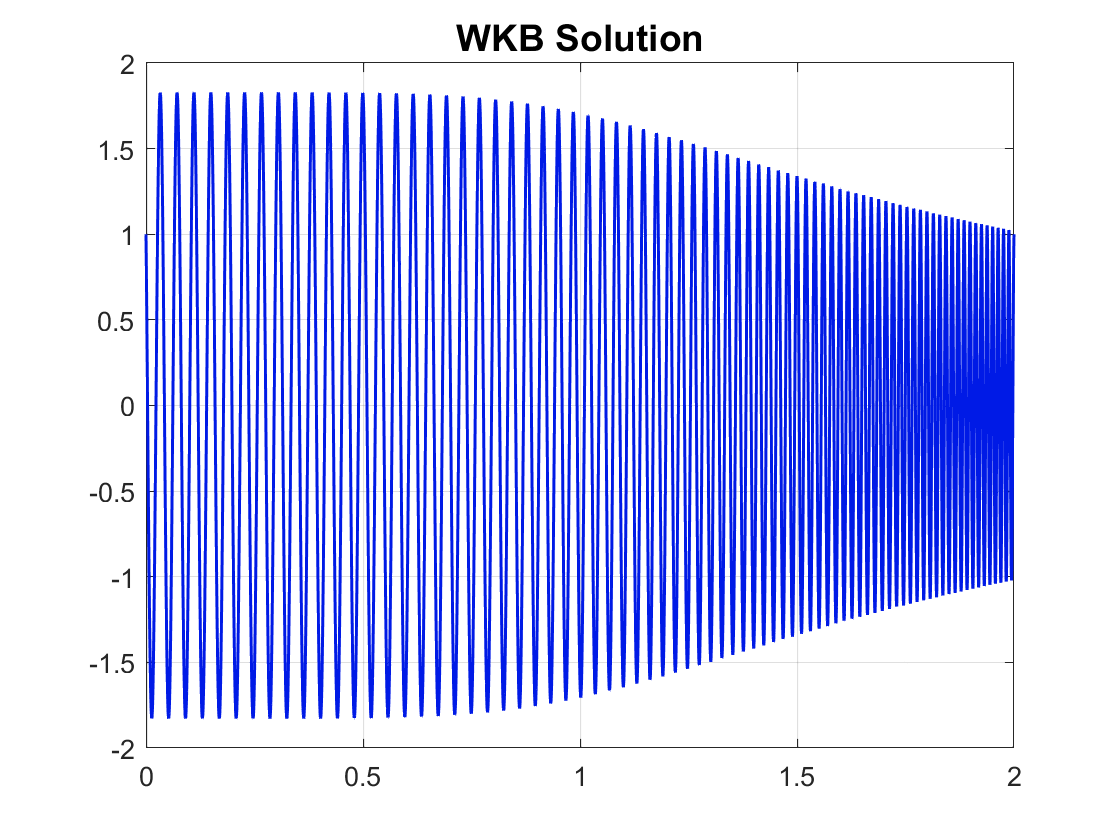}}
\subfigure[Residual for $\e=1/161.5$  \label{fig:WKBCex2res}]{ \includegraphics[width=0.45\textwidth]{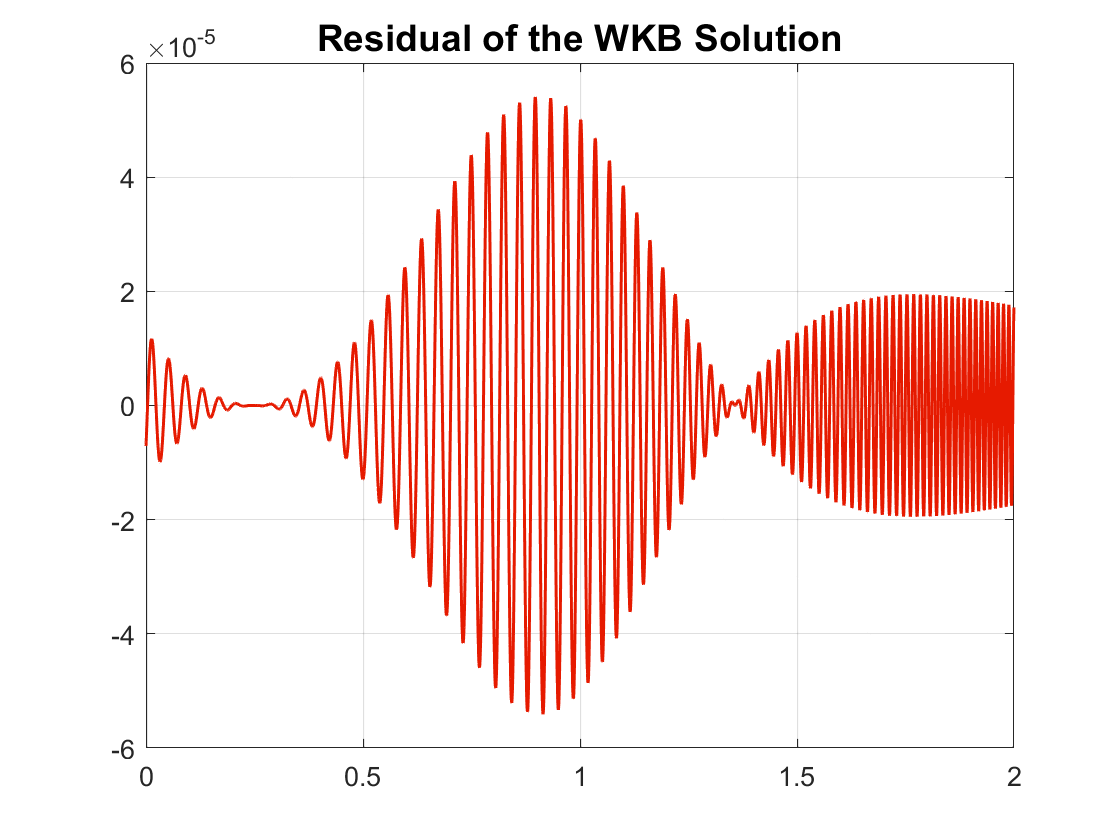}}
\\
\subfigure[Forward error  \label{fig:WKBCex2fe}]{ \includegraphics[width=0.45\textwidth]{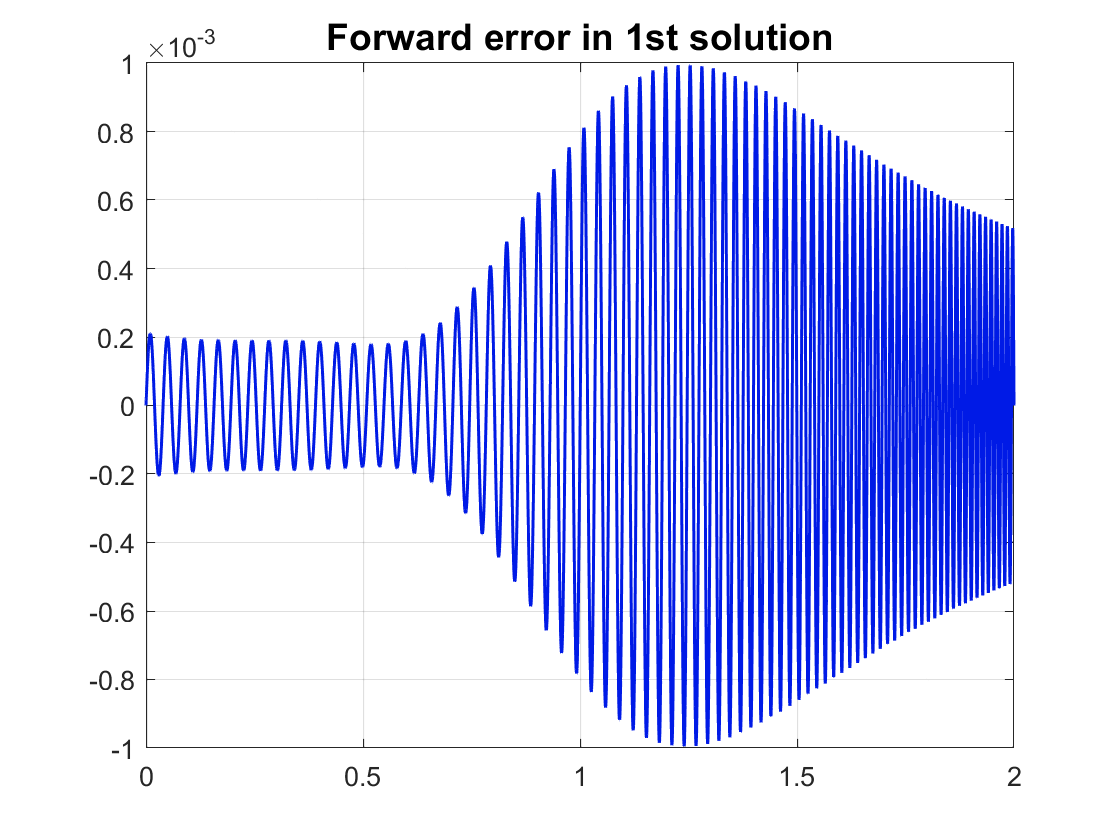}}
\subfigure[Iterated residual  \label{fig:WKBCex2iter}]{ \includegraphics[width=0.45\textwidth]{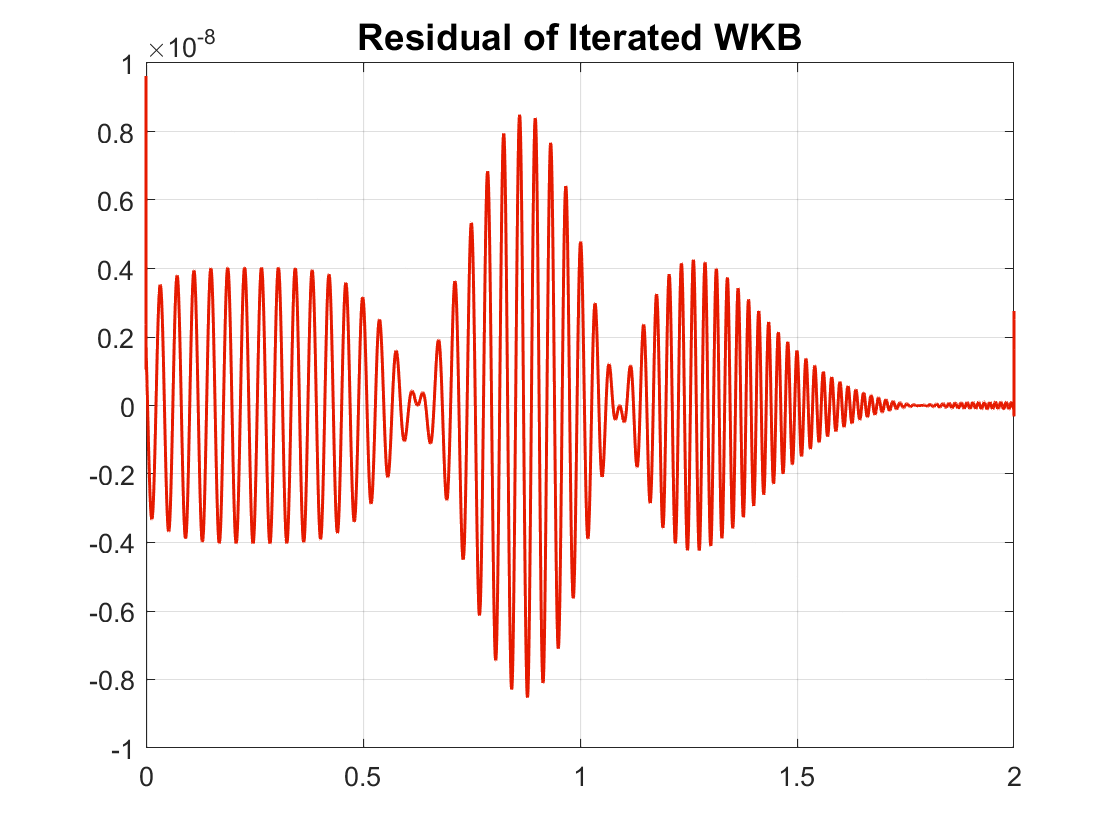}}
\caption{\label{fig:WKBCex2g}
Solution by the hybrid method in Chebfun of $\e^2 y'' +(1+(x-1/4)^4)y = 0$, $y(0)=1$ and $y(2)=1$, for $\e = 1/161.5$. Top left: the basic solution. Top right: the residual in the basic solution. Bottom left: the difference between the basic solution and the once iterated solution, which is an estimate of the forward error. Bottom right: the residual in the once iterated solution.}

\end{figure}

\begin{remark}
    We did not use the \textsl{relative} residual with Chebfun, because the exponential of a chebfun is itself converted to another chebfun, automatically. Similarly the sine of a chebfun and the cosine of a chebfun are replaced by other chebfuns.  Normally, this automatic simplification is a strength of Chebfun, but in this case it spoils the precise location of zeros of oscillatory solutions, and makes the division $y''/y$ problematic. So we use absolute residuals instead.

    If we use the Chebyshev method in Maple instead, then of course it is more complicated to have to handle all the details ourselves, but an advantage appears when using the exponential (or sine or cosine) of a Chebyshev series and not a Chebyshev series of that: the residuals will be computed correctly.
\end{remark}
We may also compare to the Chebop system of solving differential equations~\cite{driscoll2008chebop}.  The problems we solve here with WKB can also be solved using Chebop, which of course is a general-purpose tool.  
Whether we plot the symbolic Green's function, the WKB Green's function, or the chebop Green's function, all are similar. Fast as chebops are, though, the WKB approximation in Chebfun is much faster.  Of course that's because in comparison one can do so much more with chebops than just solve ODE of this particular type.
\section{Concluding remarks}
WKB provides a method for solving Schr\"odinger-type second order linear equations that complements numerical methods.
Accurate numerical solution of some of these boundary-value problems using a general-purpose code such as \lstinline{bvp5c}~\cite{kierzenka2008bvp} or \lstinline{COLNEW}~\cite{bader1987new} or \lstinline{MIRKDC}~\cite{shampine2006user} typically becomes unaffordable as $\e \to 0^+$. In that case, the WKB method provides a useful complement, as is well-known.  The observation of this paper that the WKB method provides exact solutions of nearby problems of the same type seems to us to be a very useful tool in the analysis of such solutions.

The analysis presented here is a \textsl{structured} backward error analysis.  The WKB method gives the \textsl{exact} solution to a problem in the same class with a potential perturbed by $O(\e^2)$. A large structured backward error is a strong indication that the computation cannot be trusted; this happens, for instance, at places where $Q$ is not differentiable. 

The \textsl{unstructured} condition number of the problem is given by the Green's function.  
Residual accuracy increases as $\e \to 0$, whereas numerical methods find the problem increasingly difficult as $\e \to 0$.
Some existing error analyses come remarkably close to a structured backward error analysis, but all of them that we are aware of miss the point, even if just barely. We believe that starting with an infinite series obscures the remarkable fact that the approximation from physical optics gives the exact solution to a problem with a different potential.

The hybrid method using polynomial approximation for $\sqrt{Q}$ and its integral---which we believe is new---would probably only occur naturally to someone who was already thinking about backward error. 
Once one realizes that the WKB process gives the exact result for a perturbed potential, then perturbing the potential a tiny bit more becomes perfectly acceptable.  In the method discussed here, the square root of potential $\sqrt{Q(\xi)}$ is approximated by Chebyshev polynomials, or, in Chebfun, by piecewise Chebyshev polynomials.  This makes integration straightforward, whenever $\sqrt{Q}$ can be well-approximated by a piecewise polynomial.  This is particularly convenient in Chebfun, but it can be done in Maple as well.  


If, however, $\sqrt{Q}$ contains singularities or derivative singularities, then polynomial approximation is suboptimal. In that case greater care must be taken.
Indeed, because the residual in the WKB method contains a term $Q''/(4Q)$ the residual can be unbounded if $Q$ contains singularities, and since the Green's function will be $O(1/\e)$ one expects that even with great care one may not get good results.

The WKB method is already a practical method wherever hand computation of the bottleneck integral of $\sqrt{Q}$ can be carried out.  Computer algebra systems such as Maple make that bottleneck less of a problem, by virtue of the strength of their integration capabilities.  Nonetheless even that strength can sometimes be unsatisfactory when the complicated expressions that are returned for that bottleneck integral require special handling.  In such cases, the hybrid method removes that obstacle and gives access to the good asymptotic approximation of the WKB method while retaining the simplicity of polynomial approximation using (for example) Chebyshev polynomials.  The hybrid method is particularly suitable for use in the Chebfun system, but works quite well in Maple.  
Using the WKB method iteratively with Algorithm~\ref{alg:IterativeWKB} (which we also believe is new) is also possible with the hybrid Chebyshev integration; indeed, the bottleneck in that iteration is removed completely by using Chebyshev approximation.    

Finally, we note in passing that the Langer formula, which gives a uniform approximation to the solution of problems with simple turning points, also has a simple relative residual formula, namely $\e^2 R(x)$ where
\begin{equation}
      R(x) = \frac{5 \left(\frac{d}{d x}Q \! \left(x \right)\right)^{2}}{16 Q \! \left(x \right)^{2}}-\frac{\frac{d^{2}}{d x^{2}}Q \! \left(x \right)}{4 Q \! \left(x \right)}-\frac{5 Q \! \left(x \right)}{36 \left({\textcolor{black}{\int}}_{\!\!\!0}^{x}\sqrt{Q \! \left(\xi \right)}\textcolor{black}{d}\xi \right)^{2}}\>.
\end{equation}
[A proof is given in~\cite{CorlessFillion2025}.] This suggests that we may iterate to improve the WKB approximation near turning points, as well. Preliminary experiments with this are encouraging.

\providecommand{\noopsort}[1]{}


\end{document}